\documentclass[12pt]{article}
\usepackage[hmargin={1.75truecm,1.90truecm},vmargin={2truecm,2truecm}]{geometry}
\usepackage[utf8]{inputenc}
\usepackage{hyperref}
\usepackage{amsmath}
\usepackage{amsthm}
\usepackage{amsfonts}
\usepackage{amssymb}
\usepackage{graphicx}
\usepackage{xcolor}
\usepackage{arydshln}
\usepackage[affil-it]{authblk}
\usepackage{natbib}
\usepackage{url}
\usepackage{afterpage}
\usepackage{amsmath} 
\usepackage{amssymb} 
\usepackage{multirow}
\usepackage{multicol}
\usepackage{booktabs}
\usepackage{orcidlink}
\usepackage[linesnumbered,ruled,vlined]{algorithm2e}

\SetCommentSty{AlgoCommentStyle}
\usepackage{rotating} 
\usepackage[linesnumbered,ruled]{algorithm2e} 
\usepackage{amsthm}
\usepackage{caption}
\usepackage{subcaption}
\usepackage{cleveref}
\usepackage{ragged2e}
\usepackage{todonotes}
\usepackage{tikz}
\usepackage{tikz-3dplot}
\usepackage{tikzscale}
\usepackage{pdflscape} 
\usepackage{float}
\usetikzlibrary{arrows}
\usetikzlibrary{shapes}
\usetikzlibrary{snakes}
\usetikzlibrary{patterns}
\usetikzlibrary{backgrounds,topaths}   
\usepackage{lipsum}
\usepackage{threeparttable}

\newcommand{\ROMAN}[1]{\uppercase\expandafter{\romannumeral#1}}
\newcommand{\rev}[1]{{\color{black}#1}}

\newcommand{\roundtwo}[1]{{\color{black}#1}}

\newtheorem{theorem}{Theorem}

\newtheorem{observation}[theorem]{Observation}
\newtheorem{remark}[theorem]{Remark}
\newtheorem{proposition}[theorem]{Proposition}
\newtheorem{property}[theorem]{Property}

\crefname{theorem}{Theorem}{Theorems}
\crefname{example}{Example}{Examples}
\crefname{observation}{Observation}{Observations}
\crefname{remark}{Remark}{Remarks}
\crefname{proposition}{Proposition}{Propositions}
\crefname{property}{Property}{Propertys}
\crefname{lemma}{Lemma}{Lemmas}
\crefname{corollary}{Corollary}{Corollaries}
\crefname{algocf}{Algorithm}{Algorithms}	
\crefname{table}{Table}{Tables}	
\crefname{figure}{Figure}{Figures}
\crefname{algorithm}{Algorithm}{Algorithms}
\crefname{section}{Section}{Sections}

\usepackage{tikz}
\usepackage{tikz-3dplot}
\usepackage{tikzscale}
\usetikzlibrary{arrows}
\usetikzlibrary{shapes}
\usetikzlibrary{snakes}
\usetikzlibrary{patterns}
\usetikzlibrary{backgrounds,topaths}

\newcommand{\SCP}{{SCP}\xspace}

\newcommand{\PSCP}{{PSCP}\xspace}

\newcommand{\MIR}{{MIR}\xspace}

\newcommand{\LP}{\text{LP}\xspace}

\newcommand{\CPX}{\texttt{CPX}\xspace}
\newcommand{\solver}[1]{\textsc{#1}\xspace}
\newcommand{\CPLEX}{\solver{CPLEX}}

\newcommand{\MIP}{{MIP}\xspace}
\newcommand{\BD}{{BD}\xspace}
\newcommand{\SBD}{\texttt{BD}\xspace}
\newcommand{\RBD}{\texttt{RBD}\xspace}
\newcommand{\AUTO}{\texttt{AUTO-BD}\xspace}
\newcommand{\F}{\mathcal{F}}
\newcommand{\N}{\mathcal{N}}

\newcommand{\CP}{\mathcal{P}}
\newcommand{\CR}{\mathcal{R}}
\newcommand{\CS}{\mathcal{S}}

\newcommand{\X}{\mathcal{X}}
\renewcommand{\L}{\mathcal{L}}
\newcommand{\U}{\mathcal{U}}
\newcommand{\Y}{\mathcal{Y}}

\newcommand{\setN}{[n]}
\newcommand{\setM}{[m]}
\newcommand{\setS}{[s]}

\newcommand{\T}{\texttt{T}\xspace}
\newcommand{\NODE}{\texttt{N}\xspace}
\newcommand{\G}{\texttt{G}\xspace}
\newcommand{\Tone}{\texttt{$\text{T}_1$}\xspace}
\newcommand{\Ttwo}{\texttt{$\text{T}_2$}\xspace}
\newcommand{\CPU}{\text{CPU}\xspace}

\newcommand{\BasicSBD}{\texttt{BasicBD}\xspace}
\newcommand{\BasicSBDI}{\texttt{BasicBD$+$IC}\xspace}
\newcommand{\BasicSBDIH}{\texttt{BasicBD$+$IC$+$H}\xspace}
\newcommand{\BasicSBDIHV}{\texttt{BasicBD$+$IC$+$H$+$MIR}\xspace}

\providecommand{\keywords}[1]{\small \quad \quad \textbf{Keywords:} #1}

\newcommand{\BnC}{{B\&C}\xspace}
\newcommand{\RENS}{{RENS}\xspace}

\newcommand{\CCPs}{{CCPs}\xspace}

\DeclareMathOperator{\supp}{supp}

\usepackage{multirow}
\usepackage{bm}
\usepackage[T1]{fontenc}
\usepackage{caption}
\usepackage{longtable}
\usepackage{color,xcolor}
\usepackage{rotating} 
\usepackage{makecell}

\usepackage{graphicx}

\def\headlinea{\specialrule{0.5pt}{0pt}{3pt}}
\def\headlineb{\specialrule{0.5pt}{1pt}{3pt}}

\def\footlinea{\specialrule{0.5pt}{1pt}{1pt}}
\def\footlineb{\specialrule{0.5pt}{1pt}{0pt}}
\def\endline{\specialrule{0.5pt}{3pt}{0pt}}

\newcommand{\myarraystretch}{\renewcommand{\arraystretch}{0.9}}
\newcommand{\mytabcolsepeach}{\setlength{\tabcolsep}{9pt}}

\title{Benders decomposition for the large-scale probabilistic set covering problem}

\author[a,b,c]{Jie Liang}
\author[a]{Cheng-Yang Yu}
\author[a]{Wei Lv}
\author[a,b]{Wei-Kun Chen}
\author[c,d]{Yu-Hong Dai}

\affil[a]{\small School of Mathematics and Statistics, Beijing Institute of Technology, Beijing 100081, China\\\textit{\{liangjie,yuchengyang,lvwei,chenweikun\}@bit.edu.cn}}
\affil[b]{\small State Key Laboratory of Cryptology, P. O. Box 5159, Beijing, 100878, China}
\affil[c]{\small Sichuan Aerospace Chuannan Initiating Explosive Technology Limited, Luzhou 646000, China}
\affil[d]{\small Academy of Mathematics and Systems Science, Chinese Academy of Sciences, Beijing 100190, China}
\affil[e]{\small School of Mathematical Sciences, University of Chinese Academy of Sciences, Beijing 100049, China\\\textit{dyh@lsec.cc.ac.cn}}
\date{}

\begin{document}
	
	\maketitle
	
	\begin{abstract}
		In this paper, we consider a probabilistic set covering problem (\PSCP) in which each $0$-$1$ row of the constraint matrix is random with a finite discrete distribution, 
		and the objective is to minimize the total cost of the selected columns such that each row is covered with a prespecified probability.
		{We develop an effective decomposition algorithm for the \PSCP based on the Benders reformulation of a standard mixed integer programming (\MIP) formulation.}
		The proposed Benders decomposition (\BD) algorithm {enjoys two key advantages: (i) the number of variables in the underlying Benders reformulation is equal to the number of columns but independent of the number of scenarios of the random data;}
		and (ii) 
		the Benders feasibility cuts can be separated by an efficient  polynomial-time algorithm, which makes it particularly suitable for solving large-scale \PSCP{s}.
		We enhance the \BD algorithm by using initial cuts to strengthen the relaxed master problem, implementing an effective heuristic procedure to find high-quality feasible solutions, and adding mixed integer rounding enhanced Benders feasibility cuts to tighten the problem formulation.
		Numerical results demonstrate the efficiency of the proposed \BD algorithm over a state-of-the-art \MIP solver.
		Moreover, the proposed \BD algorithm can efficiently identify optimal solutions for instances with up to $500$ rows, $5000$ columns, and $2000$ scenarios of the random rows.
		\vspace{8pt} \\
		\keywords{Large-scale optimization $\cdot$ Benders decomposition $\cdot$ probabilistic set covering problem $\cdot$ mixed integer programming Location}
	\end{abstract}

	\section{Introduction}

Given a 0-1  $m\times n$ matrix $A$ and an $n$-dimensional vector $c$ representing the cost of the columns, 
the set covering problem (\SCP) is to seek a subset of columns such that the sum of the costs of the selected columns is minimized while covering 
all rows of matrix $A$.
The mathematical formulation for the \SCP  \citep{Toregas1971} can be written as 
\begin{equation}\label{SCP}\tag{SCP}
	\min \left \{c^\top x\, : \,A_i x \geq 1,~ \forall~ i \in \setM, ~ x \in \{0,1\}^n\right\},
\end{equation}
where $A_i$ is the $i$-th row of matrix $A$, variable $x_j$ takes value $1$ if column $j$ is selected and $0$ otherwise, and $[m]:=\{1,\ldots, m\}$ (throughout the paper, for a nonnegative integer $k$, we denote $[k] = \{1, \dots, k\}$ where $[k] = \varnothing$ if $k = 0$).
The \SCP is one of the fundamental combinatorial optimization problems, and serves as a building block in various applications, such as facility location, manufacturing, vehicle routing, railway and airline crew scheduling, among many others.
We refer to the recent surveys \cite{Farahani2012,Garcia2019} and the references therein for a detailed discussion of the \SCP and its various exact and heuristic algorithms.

In this paper, we consider the probabilistic set covering problem (\PSCP) in which each ${A_i}$ is a Bernoulli random vector,
 and the objective is to find a minimum-cost
collection of columns that covers each row $i$ with a probability at least $1 - \epsilon_i$:
\begin{equation}\label{PSCP}\tag{PSCP}
	\min \left \{c^\top x\, : \,\mathbb{P}\left\{A_i x \geq 1\right\} \geq 1- \epsilon_i,~ \forall~ i \in \setM, ~ x \in \{0,1\}^n\right\}.
\end{equation}
Here, $\epsilon_i \in (0,1)$, $i \in [m]$, are prespecified  reliability labels.
As opposed to the \SCP, the \PSCP can effectively capture the uncertainty of the ``coverage'' arising in many applications. 
For example,
in the surveillance problem, there is a nontrivial probability that sensor $j$ will be able to detect target $i$ \citep{Ahmed2013};
in the location of emergency services, the event of a facility located at location $j$ to provide service to a customer $i$ is uncertain due to the variability of the traveling time \citep{Beraldi2010};
in the airline crew-scheduling problem, the event of a crew covering a flight is uncertain due to the possible crew no-shows \citep{Fischetti2012}.

 Observe that \eqref{PSCP} is an individual probabilistic programming formulation \citep{Charnes1958} where each individual constraint is required to be satisfied with a prespecified probability. 
 Another approach to capture the uncertainty of the ``coverage'' in the \SCP is to use the joint probabilistic programming framework; 
 see, e.g., \cite{Ruszczynski2002,Luedtke2014,Lejeune2016,Hu2022,Kucukyavuz2022,Jiang2022} among many of them. 
 In this framework, the $m$ probabilistic  constraints in  \eqref{PSCP} are replaced by a \emph{joint} probabilistic constraint $\mathbb{P}(Ax \geq \boldsymbol{1}) \geq 1-\epsilon$ ($\epsilon \in (0,1)$), where $\boldsymbol{1}$ is the $m$-dimensional all-ones vector.
 Compared with the joint probabilistic programming framework, the individual probabilistic programming framework allows for the inhomogeneous reliability levels $\epsilon_i$ for different rows \cite[Chapter 5]{Haneveld2020} and thus it is more flexible in many applications.
 For example,  in the aforementioned applications, some targets, customers, or flights are usually much important than others, and thus it is reasonable to provide different reliability levels $\epsilon_i$ for the detection or coverage of them.
 Due to this advantage,  \eqref{PSCP}, the individual probabilistic programming formulation, has received many attentions in the literature.
For instance, \cite{Hwang2004} considered the case where the random coefficients $a_{ij}$ of matrix $A$ are independent. 
\cite{Fischetti2012} investigated the case where the random columns of matrix $A$ are independent. 
\cite{Ahmed2013} studied the case where the random coefficients in each row of matrix $A$ are conditionally independent, 
meaning that the coefficients $a_{ij}$ for $j \in [n]$ are independent when conditioned on some exogenous factors but may be correlated without such conditioning.

Unlike the existing works, we do not assume any independent or special correlated distributions of $A_i$ in \eqref{PSCP}.
Instead, we consider general finite discrete distributions of $A_i$, that is, each $A_i$ has a finite number of scenarios.
The assumption of finite discrete distributions is not restrictive as by taking an independent Monte Carlo sampling from general distributions of $A_i$, we can obtain a problem that satisfies this assumption and can be treated as \rev{a sample average approximation (SAA)} of \eqref{PSCP} with general distributions \citep{Kleywegt2002}.
We refer to \cite{Nemirovski2006,Luedtke2008} for the theoretical and empirical evidence demonstrating that solving this sampling version can approximately solve the original problem with general distributions of $A_i$.  

\eqref{PSCP} with finite numbers of scenarios of $A_i$ can be transformed into a deterministic mixed integer programming (\MIP) formulation, known as the big-$M$ formulation  \citep{Kucukyavuz2022}, which can be solved to global optimality using general-purpose \MIP solvers.
However, due to the intrinsic NP-hardness, and particularly the huge number of binary variables and constraints (one for every scenario of $A_i$) and weak linear programming (\LP) relaxation of the big-$M$ formulation, the above approach cannot solve problem \eqref{IP-IPSCP} exactly or return a high-quality solution within a reasonable timelimit, especially when the number of scenarios is large.
\subsection{Contributions and Outlines}
The main motivation of this work is to fill this research gap, i.e., to develop an efficient  Benders decomposition (\BD) algorithm for solving large-scale \PSCP{s}. 
Two key features of the proposed \BD algorithm, are  (i) the number of variables in the Benders reformulation of the big-$M$ formulation is equal to the number of columns $n$ but independent of the number of scenarios of the random rows; and (ii) the Benders feasibility cuts can be separated by an efficient polynomial-time algorithm, which renders it particularly suitable to solve large-scale \PSCP{s}, compared with the direct use of the state-of-the-art general-purpose \MIP solvers.

We summarize two technical contributions of the paper as follows.
\begin{itemize}
	\item 
	We show that how the \PSCP relates to the partial set covering problem investigated in \cite{Daskin1999,Cordeau2019}.
	In particular, the \MIP formulation of the partial set covering problem can be seen as a special case of the big-$M$ \MIP formulation of the \PSCP.
	Based on this observation, we generalize the single-cut \BD algorithm proposed by \cite{Cordeau2019} for solving the partial set covering problem and present an efficient multi-cut \BD algorithm for solving the \PSCP.
	\item
	We propose three enhancement techniques to improve the performance of the proposed \BD algorithm.
	These techniques include using initial cuts to strengthen the relaxed master problem, implementing an effective heuristic procedure to find a high-quality feasible solution, and developing mixed integer rounding (\MIR)-enhanced Benders feasibility cuts to tighten the formulation of the \PSCP.
\end{itemize}

Numerical results demonstrate that the proposed enhancement techniques can effectively improve the formulation of the \PSCP or enable to find a high-quality solution, thereby significantly improving the overall performance of the \BD algorithm; 
the proposed \BD algorithm significantly outperforms a state-of-the-art \MIP solver's branch-and-cut and automatic \BD algorithms.
In particular, the proposed \BD algorithm enables to quickly identify optimal solutions for \PSCP{s} with up to 500 rows, 5000 columns, and 2000 scenarios of the random rows.

The rest of the article is organized as follows.
In \cref{subsec:literaturereview}, we review the relevant literature on the \PSCP.
In \cref{sect:MIP}, we introduce the \MIP formulation for the \PSCP, discuss the relation to the partial set covering problem, and demonstrate the difficulty of solving it. 
In \cref{sect:BD}, we present the \BD algorithm for the \PSCP.
In \cref{sect:implementation}, we develop three enhancement techniques for the \BD algorithm.
In \cref{sect:numer}, we perform numerical experiments to demonstrate the efficiency of the proposed enhancement techniques and the \BD algorithm for the \PSCP.
Finally, we summarize the conclusions in \cref{sect:conclusion}.

\subsection{Literature Review}
\label{subsec:literaturereview}

	\eqref{PSCP} belongs to the class of chance-constrained programs (\CCPs),
	\rev{which have been extensively studied in the literature; see \cite{Ahmed2008,Birge2011,Kucukyavuz2022}  for detailed reviews of \CCPs. 
	From a computational perspective, \CCPs are challenging to solve due to the following two difficulties. 
	First, given a point $x$, checking the feasibility of  $x$ could be computationally demanding. 
	Second, the feasible region defined by the probabilistic constraints is generally nonconvex.
	To address the above two challenges, the SAA approach can be used where an approximation problem based on an independent Monte Carlo sample of the random data is solved to obtain a high-quality solution of the original problem. 
	The main advantage of the SAA approach is its generality as it only requires to be able to sample from this distribution of the random data but does not require knowledge of the distribution.
	In this paper, we also use the SAA approach to tackle \PSCP{s} with a general  distribution of the random data $\{A_i\}$.
	It should be mentioned that in addition to the SAA approach, various conservative approximation approaches have been proposed in the literature; see, e.g.,\cite{Nemirovski2007,Ben2009,Xie2020,Jiang2022}.
	These approaches attempt to approximate the original problem as a convex optimization problem, which is computationally tractable and can usually produce high-quality feasible solutions.
	Unfortunately, for \CCPs with discrete variables (such as the considered \PSCP), the above approximations may still be difficult to solve as the resultant feasible regions are still nonconvex.}

\rev{For the probabilistic version of \SCP{s},}
let us consider a more general case of the covering constraints $A_i x\geq \xi_i$ , where $A_i \in \{0,1\}^n$ and $\xi_i \in \{0,1\}$, $i \in [m]$, may be random.
Based on the uncertainty in the covering constraints $A_i x\geq \xi_i$, $i \in [m]$ (or equivalently, $Ax \geq \xi$),
the  \PSCP variants can generally be classified into two categories: (i) uncertainty arises in the right-hand side $\xi \in \{0,1\}^m$; and (ii) uncertainty arises in the constraint matrix $A\in \{0,1\}^{m\times n}$.
\cite{Patrizia2002} first studied the joint probabilistic set covering problem, in which there is a single joint chance constraint and the uncertainty arises in the right-hand side $\xi$:
\begin{equation} 
	\label{ccprhs}
\min \left \{c^\top x\, : \,\mathbb{P}\left\{A x \geq  \xi\right\} \geq 1- \epsilon,~ x \in \{0,1\}^n\right\}.
\end{equation} 
{The authors developed a specialized branch-and-bound algorithm, which involves enumerating the so-called $(1-\epsilon)$-efficient point and solving a deterministic \SCP for each $(1-\epsilon)$-efficient point.}
Subsequently, \cite{Saxena2010} developed an equivalent \MIP formulation for problem \eqref{ccprhs} and derived polarity cuts to improve the computational efficiency.
\cite{Luedtke2008} studied problem \eqref{ccprhs} with a finite discrete distribution of the random vector $\xi$ and proposed a preprocessing technique to improve the performance of employing \MIP solvers in solving the equivalent big-$M$ formulation.
\rev{\cite{Chen2024} further proposed an efficient \BD algorithm for solving large-scale problems.}

For the problem with uncertainty on the constraint matrix $A$, as considered in this paper,  \cite{Hwang2004} and \cite{Fischetti2012}  investigated the cases in which the random coefficients $a_{ij}$ of matrix $A$ and the random columns of matrix $A$ are independent, respectively. 
In both cases, \eqref{PSCP} can be equivalently formulated as a compact \MIP problem and thus can be solved to optimality by general-purpose \MIP solvers.
\cite{Ahmed2013} established a compact mixed integer nonlinear programming formulation for \eqref{PSCP} where the random coefficients in each row of matrix $A$ are conditionally independent.
\cite{Beraldi2010} proposed a specialized branch-and-bound algorithm to solve the joint \PSCP in which there is a single joint chance constraint and the random matrix $A$ has a finite discrete distribution of $A$, that is, $\min \left \{c^\top x\, : \,\mathbb{P}\left\{A x \geq  1\right\} \geq 1- \epsilon,~ x \in \{0,1\}^n\right\}$.
\cite{Song2013} developed several branch-and-cut approaches for solving the joint \PSCP reformulated from the reliable network design problem.
\cite{Wu2019} investigated the joint probabilistic partial set covering problem and developed a delay constraint generation algorithm.
\rev{\cite{Ahmed2013,Lutter2017,Shen2023,Jiang2024}} investigated various robust \PSCP{s}.
We refer to \cite{Azizi2022} for more \PSCP  variants and their applications in location problems.

Recently, the \BD algorithm has been applied to solve various large-scale \SCP variants,
including the partial set covering and maximal covering location problems \citep{Cordeau2019},
the probabilistic partial set covering problem \citep{Wu2019},
the maximum availability service
facility location problem \citep{Muffak2023}, 
the capacitated and uncapacitated facility location problems \citep{Fischetti2016,Fischetti2017,Weninger2023}, and the location problems with interconnected facilities \citep{Kuzbakov2023}.
We refer to \cite{Rahmaniani2017} for a detailed survey of the \BD algorithm and its application in many optimization problems.
Special attention should be paid to the partial set covering problem considered in \cite{Cordeau2019}, which minimizes the cost of the selected columns while forcing a certain amount of (weighted) rows to be covered.
\cite{Cordeau2019} proposed an efficient \BD algorithm that is capable of solving instances with 100 columns and up to 40 million rows.
In this paper, we show that the \MIP formulation of this problem can be seen as a special case of the big-$M$ formulation of \eqref{PSCP} (with $m=1$), and extend the \BD algorithm of \cite{Cordeau2019} for this special case to the general case. 
Moreover, to overcome the weakness of the \LP relaxation that appears in the general problem \eqref{PSCP}, we propose three enhancement techniques that significantly improve the performance of the \BD algorithm.

	\section{Mixed integer programming formulation}\label{sect:MIP}

Consider \eqref{PSCP} with finite discrete distributions of random {vectors $A_i$}, $i \in \setM$, that
is, for each $i$, there exist vectors ${A}_i^\omega \in \{0,1\}^n$ and $p_i^\omega>0$,  $\omega\in [s] \footnote{For simplicity of exposition, here we assume that the numbers of scenarios of $A_i$ are identical for all $i \in [m]$. The extension to the case that the numbers of scenarios of $A_i$ are different for different $i$ is straightforward.}$, such that 
\begin{equation*}
	\mathbb{P}\{{A_i} = {A}_i^\omega\} = p_i^\omega,~\forall~\omega \in [s],~ \text{and}~\sum_{\omega \in [s]}p_i^\omega = 1.
\end{equation*}
We introduce binary variables $z_i^\omega$ for $i \in \setM$ and $\omega \in \setS$,
where $z_i^\omega = 1$ guarantees that in scenario $\omega$, ${A}_i^\omega x \geq 1$ holds.
Then, \eqref{PSCP} can be written as the following \MIP formulation:
\begin{subequations}\label{IP-IPSCP}
	\begin{align}
		\min ~ &c^\top x \nonumber\\
		\text{s.t.}~~  & A_i^\omega x \geq z_i^\omega, & \forall~i \in \setM, ~\omega \in \setS, \label{ip_ipscp_cons1}\\
		& \sum_{\omega \in \setS}{p_i^\omega} z_i^\omega \geq 1 - \epsilon_i,& \forall~i \in \setM, \label{ip_ipscp_cons2}\\ 
		& x_j \in \{0,1\}, & \forall~j \in \setN,  \label{ip_ipscp_x}\\
		& z_i^\omega \in \{0, 1\}, &\forall ~i \in \setM, ~\omega \in \setS. \label{ip_ipscp_cons3}
	\end{align}
\end{subequations}
If $z_i^\omega = 1$, \eqref{ip_ipscp_cons1} enforces $A_i^\omega x \geq 1$; otherwise, it reduces to the trivial inequality $A_i^\omega x \geq 0$. 
Constraints \eqref{ip_ipscp_cons1} are referred to as big-$M$ constraints where the big-$M$ coefficients are given by $1$; see \cite{Kucukyavuz2022}.
Notice that constraints \eqref{ip_ipscp_cons1}, \eqref{ip_ipscp_cons2}, and \eqref{ip_ipscp_cons3} ensure that the probability $ \mathbb{P}\{A_i x \geq 1\} $ is at least $1-\epsilon_i$.

\begin{remark}
	When $m=1$, there exists only a single constraint in \eqref{ip_ipscp_cons2}.
	In this case, problem \eqref{IP-IPSCP} is equivalent to the partial set covering problem \citep{Daskin1999, Cordeau2019} where each $j \in [n]$ and $\omega \in [s]$ corresponds to a facility and a customer, respectively, and constraint \eqref{ip_ipscp_cons2} reduces to the so-called coverage constraint.
\end{remark}

Problem \eqref{IP-IPSCP} is an \MIP problem which can be solved to optimality using the state-of-the-art \MIP solvers. 
However, the following two difficulties make it hard to solve large-scale \PSCP{s} efficiently using the above approach.
First, as opposed to \eqref{SCP} which involves only $m$ constraints and $n$ variables, problem \eqref{IP-IPSCP} involves $m s + m$ constraints and $m s + n$ variables, resulting in a problem size that is one order of magnitude larger.
The large problem size leads to a large \LP relaxation and a large search space, thereby making it hard for the \MIP solvers to solve \PSCP{s} efficiently, especially for the case with a large number of scenarios $s$.
To overcome this weakness, we present the following result stating that the integrality constraints on binary variables $z_i^\omega$ in problem  \eqref{IP-IPSCP} can be relaxed; see \cite{Cordeau2019} for a similar result for the partial set covering problem.
\begin{property}\label{proprelax1}
	The optimal value of problem \eqref{IP-IPSCP} does not change if the integrality constraints \eqref{ip_ipscp_cons3} are relaxed into 
	\begin{equation}
		\label{cont}
		0 \leq z_i^\omega \leq 1, ~\forall~i\in \setM, ~\omega \in \setS.
	\end{equation}
\end{property}
\begin{proof}
	Let $(\bar{x}, \bar{z})$ be an optimal solution of problem \eqref{IP-IPSCP} with  \eqref{ip_ipscp_cons3} replaced by \eqref{cont} and 
	$\F =\{ (i,\omega) \in \setM \times \setS\, : \, 0 < \bar{z}_i^\omega < 1 \}$.
	If $\F=\varnothing$, the statement follows. 
	Otherwise, as $\bar{x},~A_i^\omega\in \{0,1\}^n$, $A_i^\omega \bar{x}$, $(i, \omega) \in \F$, must be an integer value larger than or equal to $1$. 
	This implies that setting  $\bar{z}_i^\omega = 1$ for all $(i, \omega) \in \F$ will yield another feasible solution to problem \eqref{IP-IPSCP} with the same objective value. 
	Therefore, the statement follows as well.
\end{proof}
\noindent Based on the above property, we can bypass the first difficulty by generalizing the single-cut \BD approach for the partial set covering problem (i.e., problem \eqref{IP-IPSCP} with $m=1$) of \cite{Cordeau2019} and propose a scalable multi-cut \BD algorithm to solve problem \eqref{IP-IPSCP} (with arbitrary $m$)
where variables $z_i^\omega$ are projected out from the problem and replaced by the Benders feasibility cuts; see \cref{sect:BD} further ahead.

The second difficulty is that the \LP relaxation of problem \eqref{IP-IPSCP} could be very weak in terms of providing a  poor \LP bound, as will be shown in our experiments.
This is in sharp contrast to the special case $m=1$ in which \cite{Cordeau2019} observed that the difference between the optimal value of problem \eqref{IP-IPSCP} with $m=1$ and that of its \LP relaxation is usually very small.
The weakness of the \LP relaxation for the general case makes it challenging for the general-purposed \MIP solvers or the coming \BD approach to solve problem \eqref{IP-IPSCP} exactly or to obtain high-quality solutions.
In \cref{sect:implementation}, we will develop three enhancement techniques to overcome this weakness.

	\section{Benders decomposition}\label{sect:BD}

In this section, we first introduce an equivalent Benders reformulation for problem \eqref{IP-IPSCP} and then present the implementation details for solving the Benders reformulation including an efficient  implementation for the separation of the Benders feasibility cuts.

\subsection{Benders reformulation}\label{BD_IPSCP}

Observe that for a fixed $x\in \{0, 1\}^n$, problem \eqref{IP-IPSCP} can be decomposed into $m$ subproblems,  
each of which determines whether $x$ satisfies the probabilistic constraint {$\mathbb{P}\left\{A_i x \geq 1\right\} \geq 1- \epsilon_i$}.
This observation, combined with \cref{proprelax1}, enables to develop a multi-cut Benders reformulation of problem \eqref{IP-IPSCP}
 (also called {Benders master problem}):
\begin{equation}
	\label{mp-ipscp}
	\min_{x}\left\{ c^\top x\, : \, B_i^r(x) \geq 0,~ \forall~r \in \CR(\CP_i), ~ i \in  \setM, ~x \in \{0, 1\}^n \right\},
\end{equation}
where for each $i \in  \setM$,
$\CR(\CP_i)$ is the set of extreme rays of polyhedron $\CP_i$ defined by the dual of the {{Benders subproblem}} $i$ and $B_i^r(x) \geq 0$ refers to the corresponding {Benders feasibility cuts}.
Note that, no {Benders optimality cut} is  needed, as variables  $z_i^\omega$ do not appear in the objective function of problem \eqref{IP-IPSCP}.
Given a vector $\bar{x} \in [0,1]^n$, the {Benders subproblem} $i$ and its dual can be written as 
 \begin{equation}
	\label{sp-ipscp-1}
		\min_{z_i}\left\{ 0 \, :\, \eqref{ip_ipscp_cons2}, \, z_i^\omega \leq A_i^\omega\bar{x}, ~ 0 \leq z_i^\omega \leq 1, ~ \forall~ \omega \in \setS \right\},
\end{equation}
and 
\begin{equation}
	\label{dsp-ipscp}
 	\max_{\pi_i, ~\sigma_i,~\gamma_i} \left \{ (1 - \epsilon_i) \gamma_i - \sum_{\omega \in \setS} \left(\pi_i^\omega A_i^\omega\bar{x} + \sigma_i^\omega \right) \, : \, (\pi_i,\, \sigma_i,\, \gamma_i) \in \CP_i \right \}, 
\end{equation}
where  $\pi_i^\omega$ and $\sigma_i^\omega$, $ \omega \in \setS$, are the dual variables associated with constraints $z_i^\omega \leq A_i^\omega\bar{x}$ and $z_i^\omega \leq 1$, respectively,  
$\gamma_i$ is the dual variable associated with constraint \eqref{ip_ipscp_cons2},
and  
\begin{equation}\label{dsp-ispcp-p}
	\begin{aligned}
		\CP_i = \left \{ (\pi_i,\, \sigma_i,\, \gamma_i) \in \mathbb{R}_+^{ s} \times \mathbb{R}_+^{ s}  \times \mathbb{R}_+ \, : \, 
		\pi_i^\omega + \sigma_i^\omega \geq p_i^\omega\gamma_i, \,\forall~ \omega \in \setS \right\}.
	\end{aligned}
\end{equation}
By the \LP duality theory, if the dual subproblem \eqref{dsp-ipscp} is unbounded, i.e., $\CP_i$ has an extreme ray $(\hat{\pi}_i,\, \hat{\sigma}_i,\, \hat{\gamma}_i)$ such that 
	$(1 - \epsilon_i) \hat{\gamma}_i - \sum_{\omega \in \setS} \left(\hat{\pi}_i^\omega A_i^\omega\bar{x} + \hat{\sigma}_i^\omega \right) > 0$, 
	then problem \eqref{sp-ipscp-1} is infeasible.
	Thus the {Benders feasibility cut} of subproblem $i$ violated by the infeasible point  $\bar{x} \in [0, 1]^n$ reads
	\begin{equation}\label{inifeacut_ipscp}
		\sum_{\omega \in \setS} \left(\hat{\pi}_i^\omega A_i^\omega x + \hat{\sigma}_i^\omega \right) \geq (1 - \epsilon_i) \hat{\gamma}_i .
	\end{equation}

	To solve the \LP problem \eqref{dsp-ipscp} and obtain the Benders feasibility cut \eqref{inifeacut_ipscp}, we follow \cite{Cordeau2019} to use an exact combinatorial approach.  
	In particular, let $(\pi_i, \sigma_i, \gamma_i)$ be  an optimal solution of problem  \eqref{dsp-ipscp}. 
	If $\gamma_i=0$, then we can set $\pi_i^\omega= \sigma_i^\omega=0$ for all $\omega \in [s]$ (as their objective coefficients are nonpositive), and thus
	the optimal value of problem \eqref{dsp-ipscp} is zero in this case.
	As a result, to determine whether problem  \eqref{dsp-ipscp} is unbounded, it is sufficient to enforce $\gamma_i > 0$ in problem \eqref{dsp-ipscp}. 
	By normalizing $\gamma_i = 1$ in problem  \eqref{dsp-ipscp}, we obtain
	\begin{equation}
		\label{dsp-ipscp2}
		\max_{\pi_i, ~\sigma_i}\left\{(1 - \epsilon_i) - \sum_{\omega \in \setS} \left(\pi_i^\omega A_i^\omega\bar{x} + \sigma_i^\omega \right) \, : \, \pi_i^\omega + \sigma_i^\omega \geq p_i^\omega,~\pi_i^\omega\geq 0,~ \sigma_i^\omega \geq 0, \, \forall~ \omega \in \setS \right\}.
	\end{equation}
	{Observe that $(\pi_i, \sigma_i) \in \mathbb{R}_+^{s} \times \mathbb{R}_+^{s}$ is a feasible solution of problem \eqref{dsp-ipscp2} (with an objective value of $f$) if and only if $ (t\pi_i, t\sigma_i, t) \in \mathbb{R}_+^{s} \times \mathbb{R}_+^{s} \times \mathbb{R}_{+} $ with $t > 0$ is a feasible solution of problem \eqref{dsp-ipscp} (with an objective value of $tf$).}
	Thus, problem  \eqref{dsp-ipscp} is unbounded if and only if the optimal value of problem \eqref{dsp-ipscp2} is larger than zero.
	It is easy to see that problem \eqref{dsp-ipscp2} is bounded and one of the optimal solutions is given by
	\begin{equation}\label{opt-dsp-ipscp}
		\begin{aligned}
			(\bar{\pi}_i^\omega, \bar{\sigma}_i^\omega)= \begin{cases}
				(p^\omega_i,0), \quad &\text{if} \; A_i^\omega\bar{x} \leq 1; \\
				(0, p^\omega_i), \quad &\text{otherwise},
			\end{cases}
		\quad \forall~ \omega \in \setS.	
		\end{aligned}
	\end{equation}
	Thus, the {Benders feasibility cut} \eqref{inifeacut_ipscp} reduces to
	\begin{equation}
		\label{feas-cut2}
		\sum_{\omega \in \setS, ~A_i^\omega\bar{x} \leq 1} p_i^\omega A_i^\omega x + \sum_{\omega \in \setS, ~A_i^\omega\bar{x} > 1} p_i^\omega  \geq 1 - \epsilon_i.
	\end{equation}
	Note that \eqref{feas-cut2} can also be derived by combining $A_i^\omega x\leq 1$ for $\omega \in [s]$ with $A_i^\omega \bar{x} \leq 1$, $z_i^\omega \leq 1$ for $\omega \in [s]$ with $A_i^\omega \bar{x} > 1$, and \eqref{ip_ipscp_cons2}.
	Also note that if $A_i^\omega\bar{x} = 1$, we can alternatively set $(\bar{\pi}_i^\omega, \bar{\sigma}_i^\omega)= (0, p_i^\omega)$, possibly resulting in a different Benders  feasibility cut with the same violation at $\bar{x}$.
	However, in our preliminary experiments, we observed that the current strategy in \eqref{opt-dsp-ipscp} generally enables the coming \BD algorithm to return a better \LP bound at the root node and thus makes a better overall performance.
	A similar phenomenon was also observed by \cite{Cordeau2019} in the context of solving the partial set covering problem. 
	Due to this, we decide to set $(\bar{\pi}_i^\omega, \bar{\sigma}_i^\omega)= (p_i^\omega,0)$ for $i \in [m]$ with $A_i^\omega\bar{x} = 1$.

\subsection{Implementation}
\label{closed_formula}

To implement the \BD algorithm for solving problem \eqref{mp-ipscp}, we adopt a {branch-and-cut} (\BnC) approach,  in which Benders feasibility cuts \eqref{feas-cut2} are added on the fly at each node of the search tree (constructed by the \BnC approach). 
This approach, known as \emph{Branch-and-Benders-cut} \citep{Rahmaniani2017}, can be implemented within the cut callback framework available in modern general-purpose \MIP solvers and 
has been widely applied to implement the \BD algorithm for various problems; see  \cite{Perez2014,Gendron2016,Fischetti2016,Cordeau2019,Guney2021}, among many of them.

We now discuss the separation for the Benders feasibility cuts \eqref{feas-cut2} in the \BD algorithm.
Given a point $\bar{x}\in [0,1]^n$, to determine whether there exists a violated {Benders feasibility cut} \eqref{feas-cut2} for each $i \in \setM$,
it is sufficient to compute $A_i^\omega\bar{x}$ for all $\omega \in \setS$ and test whether the condition $\sum_{\omega \in \setS, ~A_i^\omega\bar{x} \leq 1} p_i^\omega A_i^\omega\bar{x} + \sum_{\omega \in \setS, ~A_i^\omega\bar{x} > 1} p_i^\omega < 1 - \epsilon_i$
 is satisfied or not.
This provides an $\mathcal{O}\Bigl(\sum_{i \in  \setM}$ $\sum_{\omega \in \setS} |\supp(A_i^\omega)|\Bigr)$ algorithm for the separation of {Benders feasibility cuts} \eqref{feas-cut2} (for a vector $a\in \mathbb{R}^n$, we use $\supp(a)=\{i \in [n] \, : \, a_i \neq 0 \}$ to denote its support).
However, our preliminary experiments showed that 
this direct row-oriented implementation for the computations of  $A_i^\omega\bar{x}$ is too time-consuming, 
especially when the number of scenarios $s$ is large or vectors $A_i^\omega$ are dense. 
{To resolve this issue, we note that point $\bar{x}$ is usually very sparse  encountered in the \BD algorithm.
To this end, we use a \emph{column-oriented} implementation to speed up the computations of $A_i^\omega\bar{x}$, $\omega \in \setS$.}
Specifically, we can compute $A_i^\omega\bar{x}$, $\omega \in \setS$, using 
\begin{equation}
	\label{CompAx}
	\begin{bmatrix} 
		A_i^1  \bar{x} \\
		\vdots\\
		A_i^{s} 	\bar{x} 
	\end{bmatrix} =: \bm{A}_i \bar{x} = \sum_{j \in  \setN}{[\bm{A}_i]}_j \bar{x}_j =\sum_{j\in  \supp(\bar{x})}{[\bm{A}_i]}_j \bar{x}_j,
\end{equation}
with the complexity of $\mathcal{O}\left(\sum_{j \in \supp(\bar{x})} |\supp({[\bm{A}_i]}_j)|\right)$, where ${[\bm{A}_i]}_j \in \mathbb{R}^{s}$ is the $j$-th column of matrix $\bm{A}_i\in \mathbb{R}^{s \times n}$.
As  $\sum_{j \in \supp(\bar{x})} |\supp({[\bm{A}_i]}_j)|\leq  \sum_{j \in  \setN} |\supp({[\bm{A}_i]}_j)|= \sum_{\omega \in \setS} |\supp(A_i^\omega)|$, 
using \eqref{CompAx} for the computations of  $A_i^\omega\bar{x}$, $\omega \in \setS$, is significantly more efficient than the direct row-oriented implementation, especially when $\bar{x}$ is very sparse, i.e., $|\supp(\bar{x})| \ll n$.

The above $\mathcal{O}(\sum_{i \in [m]}\sum_{\omega \in [s]} |\supp(A_i^\omega)|)$ polynomial-time separation algorithm can also be used to detect whether problem  \eqref{PSCP} (or problem \eqref{IP-IPSCP}) has a feasible solution.
Indeed, letting $x, y \in \{0,1\}^n$ with $x \leq y$,  if $x$ is a feasible solution of problem \eqref{PSCP}, then $y$ must also be a feasible solution. 
Therefore, detecting whether problem  \eqref{PSCP} has a feasible solution is equivalent to checking whether the all-ones vector $\boldsymbol{1}$ is a feasible solution of problem  \eqref{PSCP}. 
The latter can be done by checking whether the Benders feasibility cut \eqref{feas-cut2} with $\bar{x}= \boldsymbol{1}$ is violated by vector $\boldsymbol{1}$ for some $i \in [m]$.
If there exist some $i \in [m]$ for which \eqref{feas-cut2} with $\bar{x}= \boldsymbol{1}$ is violated by vector $ \boldsymbol{1}$, i.e., 
\begin{equation}\label{infeascheck}
	\sum_{\omega \in [s], ~|\supp(A_i^\omega)| \geq 1} p_i^\omega < 1- \epsilon_i,
\end{equation}
then problem (\PSCP) (or problem \eqref{IP-IPSCP}) is infeasible; otherwise, problem  \eqref{PSCP} must have a feasible solution.
Condition \eqref{infeascheck} reflects that if $\{\epsilon_i\}$ are very small and $\{A_i^\omega\}$ are very sparse, then problem  \eqref{PSCP} is likely to be infeasible. 
However, as condition \eqref{infeascheck} can be quickly identified (with the complexity of $\mathcal{O}(\sum_{i \in [m]}\sum_{\omega \in [s]} |\supp(A_i^\omega)|)$), we, without loss of generality, assume that problem  \eqref{PSCP} has a feasible solution in the following.

To end of this section, we highlight two advantages of the proposed \BD algorithm as follows. 
First, in contrast to the big-$M$ formulation \eqref{IP-IPSCP} where the number of variables is ${n+ms}$ and grows linearly with the number of scenarios, the number of variables in the Benders reformulation \eqref{mp-ipscp} is only $n$ and thus  independent of the number of scenarios $s$.
Second, the Benders feasibility cuts \eqref{feas-cut2} can be separated in an $\mathcal{O}\left(\sum_{i \in  \setM}\sum_{\omega \in \setS} |\supp(A_i^\omega)|\right)$ polynomial-time algorithm, along with the above acceleration technique.
The above two advantages make the proposed \BD algorithm particularly suitable for solving large-scale \PSCP{s}, especially for those with a huge number of scenarios.

    \section{Improving performance of the Benders decomposition}
\label{sect:implementation}

In this section, we will propose three enhancement techniques to improve the performance of the proposed \BD algorithm for solving \PSCP{s}.
These techniques include adding tight initial cuts to strengthen the relaxed master problem, implementing a customized {relaxation enforced neighborhood search} (\RENS)  procedure \citep{Berthold2014} in the early stage of the \BD algorithm to find a high-quality feasible solution,
and adding \MIR-enhanced Benders feasibility cuts to tighten the \LP relaxation of formulation \eqref{mp-ipscp}.

\subsection{Initialization of relaxed master problem}
\label{subsec:initialcuts}
In general, the \BD algorithm can start with a relaxed version of the master problem \eqref{mp-ipscp} without any {Benders feasibility cut} added.
However, previous studies have demonstrated that it is computationally advantageous to add some tight initial cuts to strengthen the relaxed master problem \citep{Rahmaniani2017}. 
In the context of the \PSCP, we decide to add the Benders feasibility cuts  \eqref{feas-cut2} induced by point $x= \boldsymbol{0}$ (where $\boldsymbol{0}$ is the $n$-dimensional zero vector), i.e., 
\begin{equation}\label{initcut-pre}
\sum_{\omega \in \setS} p_i^\omega {A}_i^\omega x \geq 1 - \epsilon_i, ~\forall~ i \in  \setM.
\end{equation}
to the relaxed master problem as constraints.
It should be mentioned that adding constraints \eqref{initcut-pre} not only strengthens the relaxed master problem but also enables \MIP solvers to construct internal cuts (e.g., knapsack cover cuts) that further strengthen the \LP relaxation of formulation \eqref{mp-ipscp} and thus improve the performance of the \BD algorithm; see  \cref{subsec:techniques}.

\subsection{A customized relaxation enforced neighborhood search procedure}
\label{subsec:primalheuristic}
{One drawback of the \BD algorithm is that it never sees the complete formulation as a whole but only a part at a time, 
making it challenging to find high-quality feasible solutions in the early stage of the algorithm; see, e.g., \cite{Botton2013,Rahmaniani2017}.}
However, in practice, it is crucial to identify a high-quality feasible solution within a reasonable amount of time. 
Moreover, a high-quality feasible solution can serve as a useful upper bound that helps the \BD algorithm to prune the uninteresting nodes of the search tree (i.e., those that do not contain a feasible solution better than the incumbent).
In the following, we develop a customized \RENS procedure, proposed in \cite{Berthold2014} for generic \MIP{s}, to construct a high-quality solution for problem \eqref{mp-ipscp}.

Let $x^{\text{\LP}}$ be a feasible solution of the \LP relaxation of problem \eqref{mp-ipscp} (which can be accessed once the \BD algorithm finishes exploring the root node).
The \RENS procedure attempts to find the \emph{optimal rounding} of point $x^{\text{\LP}}$ by solving a subproblem of problem \eqref{mp-ipscp} \citep{Berthold2014}. 
This subproblem is 	defined by fixing variables $x_j=0$ for all $j \in \N^0$ and $x_j=1$ for all $j \in \N^1$ in problem \eqref{mp-ipscp}:
\begin{equation}
	\label{mp-ipscp-sp}
	\min_{x}\left\{ c^\top x\, : \, B_i^r(x) \geq 0,~ \forall~r \in \CR(\CP_i), ~ i \in  \setM,~x_j =0,~\forall~ j \in \N^0,~x_j=1,~\forall~j \in \N^1,~x \in \{0, 1\}^n \right\},
\end{equation} 
where
\begin{equation}
	\label{N01}
	\N^0 = \left \{j \in  \setN\, :\, x_j^{\text{\LP}} = 0\right\} \; \text{and}\; \N^1  = \left \{j \in  \setN\, :\, x_j^{\text{\LP}} = 1\right\},
\end{equation}
and $B_i^r(x) \geq 0,~ \forall~r \in \CR(\CP_i), ~ i \in  \setM,$ are all Benders feasibility cuts of the form \eqref{feas-cut2}.
As this subproblem is a restriction of problem \eqref{mp-ipscp}, it can still be solved by the \BD algorithm. 
Moreover, 
solving subproblem \eqref{mp-ipscp-sp} always returns a feasible solution.
\begin{proposition}
	\label{remark:heur}
	Let $x^\LP$ be a feasible solution of the \LP relaxation of problem \eqref{mp-ipscp}, and $\N^0$ and $\N^1$ be defined as in \eqref{N01}. 
	Then subproblem \eqref{mp-ipscp-sp} has a feasible solution.
\end{proposition}

\begin{proof}
As $p_i^\omega > 0$ and $A_i^\omega \in \{0,1\}^n$, each Benders feasibility cut \eqref{feas-cut2} must have the form of $\pi^\top x \geq \pi_0$ where $\pi \in \mathbb{R}_+^n$.
	Therefore, problems \eqref{mp-ipscp} and \eqref{mp-ipscp-sp} can be written as the forms of 
	\begin{align}\label{tmp}
		& \min \left\{ c^\top x \, : \, Cx \geq b, ~x \in \{0,1\}^n\right\},  \\
		& \min \left\{ c^\top x \, : \, Cx \geq b, ~x \in \{0,1\}^n, ~x_j= 0, ~\forall~j \in \N^0,~x_j = 1,~\forall~j \in \N^1\right\}, \label{tmp2}
	\end{align}
	with a nonnegative constraint matrix $C\geq 0$.
	Since $x^{\LP}$ is a feasible solution of the \LP relaxation of problem \eqref{mp-ipscp} (or equivalently, problem \eqref{tmp}), $Cx^{\LP} \geq b$ and $0 \leq x^{\LP} \leq 1$ must hold.
	Letting $\bar{x}= \lceil x^{\LP}\rceil$, then $\bar{x} \in \{0,1\}^n$ and $C\bar{x}\geq Cx^{\LP} \geq b$; and by $x_j^{\LP} =0$ for $j \in \N^0$ and $x_j^{\LP}=1$ for $j \in \N^1$, we must have $\bar{x}_j=0$ for $j \in \N^0$ and $\bar{x}_j=1$ for  $j \in \N^1$.
	This implies that $\bar{x}$ is a feasible solution of problem \eqref{tmp2}, i.e., problem \eqref{mp-ipscp-sp}.
	The proof is complete.
\end{proof}

\noindent \rev{\cref{remark:heur} implies that (i) problem \eqref{mp-ipscp} is feasible if and only if its \LP relaxation is feasible; and (ii) the \RENS procedure}
enjoys a favorable feature, that is, it will always return a feasible solution of the original problem \eqref{mp-ipscp}.
Indeed, as will be demonstrated in \cref{subsec:techniques}, a high-quality feasible solution of problem \eqref{mp-ipscp} is usually identified by this heuristic procedure.

Although a high-quality feasible solution can potentially be found by the above heuristic procedure, our preliminary experiments showed that the computational effort spent in solving the restricted problem \eqref{mp-ipscp-sp} can be large, especially when $\N^0$ and  $\N^1$ are small (as the problem size is large).
To save the computational effort, we enlarge the subsets $\N^0$ and $\N^1$ as
\begin{equation}
	\label{N011}
	\N^0 = \left \{j \in  \setN\, :\, x_j^{\text{\LP}} \leq \theta \right\} \; \text{and}\; \N^1  = \left \{j \in  \setN\, :\, x_j^{\text{\LP}} \geq 1-\theta \right\}, 
\end{equation}
where $\theta> 0$ is a predefined value.
The larger the $\theta$, the larger the $\N^0$ and $\N^1$ are, and thus the smaller and easier the problem \eqref{mp-ipscp-sp} is.
In our experiments, we set $\theta = 0.01$.

\subsection{\MIR-enhanced Benders feasibility cuts }\label{subsec:validineqs}

As mentioned in previous studies such as \cite{Bodur2016,Rahmaniani2020},  when constructing the Benders cuts, the \BD algorithm does not consider {the constraints (e.g., the integrality constraints) on variables $x$ in the master problem}.
This disadvantage, however, generally leads to weak Benders cuts, resulting in weak \LP relaxations and bad overall performances. 
To overcome this weakness, \cite{Cordeau2019} exploited the integrality requirements of $x$  and applied the {coefficient strengthening} \citep{Savelsbergh1994} to strengthen the Benders feasibility cuts \eqref{feas-cut2}.
Here we explore the use of the \MIR technique \citep{Nemhauser1990,Marchand2001}, which simultaneously takes the integrality and variable bound requirements of $x$ (i.e., $x\in \{0,1\}^n$) into consideration, to derive \MIR-enhanced Benders feasibility cuts (see \cite{Bodur2016} for a discussion on \MIR-enhanced Benders optimality cut in the context of solving two-stage stochastic programming problems).
We show that \MIR-enhanced Benders feasibility cuts dominate the strengthened Benders feasibility cuts in \cite{Cordeau2019}.

\subsubsection{Mixed integer rounding}
To proceed, we introduce the following result on the basic \MIR inequality, which can be found in, e.g., \cite[Proposition 8.5]{Wolsey2021}.
\begin{proposition}\label{prop-basic-MIR}
	Let $b \in \mathbb{R}$ and $\Y = \{(x, y) \in \mathbb{Z}^1 \times \mathbb{R}_+^1 \,:\,x + y \geq b\}$. 
	The basic \MIR inequality 
	$x + \frac{y}{f_b} \geq \lceil b \rceil$
	is valid for $\Y$ where $f_b = b - \lfloor b \rfloor$.
\end{proposition}
\noindent Using \cref{prop-basic-MIR}, we can derive the \MIR inequality as follows. For completeness, a proof is provided in Section 1 of the online supplement \footnote{The online supplement is available at \url{https://drive.google.com/file/d/1Cu_BmYsw6tmppf0yxCaUwBbt2ru83MzK/view?usp=sharing}.}.
\begin{proposition}
	\label{pro:mir-1}
	Consider the set $\X = \left\{x \in \{0,1\}^n\, :\, \sum_{j \in  \setN}c_{j} x_j \geq b\right\}$.
	Let $(\L,\U)$ be a partition of $ \setN$ and $\delta > 0$.
	The following \MIR inequality 
	\begin{equation}
		\label{CMIRineq}
		\sum_{j \in \L} G\left(\frac{c_j}{\delta}\right) x_j + 	\sum_{j \in \U} G\left(-\frac{c_j}{\delta}\right) (1-x_j) \geq \left \lceil \beta \right\rceil 
	\end{equation}
	is valid for $\X$, where
	\begin{equation}
		\label{gdef}
		\begin{aligned}
			&  \beta =  \frac{b-\sum_{j \in \U}c_j}{\delta},~~G(d) =  \lfloor d \rfloor+\min\left\{\frac{f_d}{f_{\beta}}, 1\right\}, \\
			& f_\beta= \beta - \lfloor \beta\rfloor, \quad f_d = d -  \lfloor d\rfloor.
		\end{aligned}
	\end{equation}
\end{proposition}

To derive the \MIR-enhanced Benders feasibility cuts for formulation \eqref{mp-ipscp}, we can present the Benders feasibility cut \eqref{feas-cut2} as $\sum_{j \in \setN}c_{j} x_j \geq b$ where
\begin{equation}
	\label{abdef}
	c_j = \sum_{\omega \in \setS , ~A_i^\omega \bar{x} \leq 1}p_i^\omega {a}_{ij}^\omega, ~j \in  \setN,~\text{and}
	~b=1 - \epsilon_i - \sum_{\omega \in \setS, ~A_i^\omega \bar{x} > 1} p_i^\omega,
\end{equation}
select the partition $(\L, \U)$ of $[n]$ and parameter $\delta>0$, and apply \cref{pro:mir-1} to construct the resultant inequality \eqref{CMIRineq}.
In order to possibly find an inequality violated by a given vector $\bar{x} \in [0,1]^n$, we choose the partition $(\L, \U)$ of $[n]$ and parameter $\delta>0$ using the following heuristic procedure.
The partition $(\L, \U)$ of $[n]$ is set to $\L =  \left\{j \in  \setN\, :\, \bar{x}_j < \frac{1}{2}\right\}$ and $\U= \left\{ j \in  \setN\, :\, \bar{x}_j \geq \frac{1}{2} \right\}$.
As for parameter $\delta$, we take the values in $\F=\left\{ |c_j|\, : \, c_j \neq 0, ~0 < \bar{x}_j < 1, ~j \in  \setN \right\}$.
We construct the \MIR inequalities for all $\delta \in \F$, and choose the one with the greatest violation. 
Here the violation of an inequality $\sum_{j \in  \setN}\pi_j x_j \geq \pi_0$ at a point $\bar{x}$ is defined as $\frac{\max\left\{\pi_0 -\sum_{j\in  \setN} \pi_j \bar{x}_j,\, 0\right\}}{||\pi||}$, where $||\cdot||$ denotes the Euclidean norm.
A similar heuristic  procedure has been used by \cite{Marchand2001} to construct violated $\leq$-\MIR inequalities for the set $\left\{x \in \mathbb{R}^p \times \mathbb{Z}^{n-p} \, :\, \sum_{j \in  \setN}c_{j} x_j \leq b,~\ell \leq x \leq u\right\}$.

\subsubsection{Comparison with the strengthened Benders feasibility cut in \cite{Cordeau2019}}

Next, we illustrate the strength of the \MIR-enhanced Benders feasibility cut \eqref{CMIRineq} for formulation \eqref{mp-ipscp} by comparing it with the strengthened Benders feasibility cut in \cite{Cordeau2019}.

We first introduce the strengthened Benders feasibility cut.
Without loss of generality, we assume $b  > 0$ where $b$ is defined in \eqref{abdef}, as otherwise, $\X=\{0,1\}^n$ and no violated inequality can be derived from $\X$. 
In addition, from the definitions of $c_j$ and $b$ in \eqref{abdef}, $p_i^\omega > 0$, $\sum_{\omega\in [s]}p_i^\omega=1$, $\epsilon_i >0$, and $a_{ij}^\omega \in \{0,1\}$, we must have 
\begin{observation}\label{abobser}
	Let $c_j $ and $b$ be defined as in \eqref{abdef}. Then $0 \leq c_j \leq 1$ for all $j \in [n]$ and $0 < b < 1$.
\end{observation}

\noindent Using the integrality and nonnegativity of variables $x_j$, \cite{Cordeau2019} modified the coefficient $c_j$ to $\min\{c_j,b\}$ in the Benders feasibility cut $\sum_{j\in[n]} c_j x_j \geq b$ and derived the strengthened Benders feasibility cut
\begin{equation}
	\label{sBDcut}
\sum_{j\in [n]} \min \{c_j, b\}x_j \geq b.
\end{equation}
\begin{proposition}\label{pro:eqmir}
	The strengthened Benders feasibility cut \eqref{sBDcut} is equivalent to the \MIR-enhanced Benders feasibility cut \eqref{CMIRineq} with $\L = [n]$, $\U = \varnothing$, and $\delta=1$.
\end{proposition}
\begin{proof}
	Suppose that $\L = [n]$, $\U = \varnothing$, and $\delta=1$.
	By \eqref{gdef} and \cref{abobser}, we have $G(c_j)= \lfloor c_j \rfloor+\min \left\{\frac{f_{c_j}}{f_b},1\right\}=\lfloor c_j \rfloor + \min \left\{\frac{c_j-\lfloor c_j \rfloor}{b-\lfloor b\rfloor}, 1 \right\}$. 
	If $c_j = 1$, then $G(c_j) = 1= \min \left\{\frac{c_j}{b},1\right\}$; otherwise, $G(c_j) = \min \left\{\frac{c_j}{b},1\right\}$.
	Together with $\lceil\beta\rceil = \lceil b \rceil=1$, this indicates that \eqref{sBDcut} is a scalar multiple of \eqref{CMIRineq} with the scalar being $b$.
As a result, inequalities \eqref{CMIRineq} and \eqref{sBDcut} are equivalent.
\end{proof}

Under certain conditions, the \MIR-enhanced Benders feasibility cut \eqref{CMIRineq} can be stronger than the strengthened Benders feasibility cut \eqref{sBDcut}.
\begin{proposition}
	\label{pro:coef-1}
	Let $\CS=\{ j \in [n]\, : \, c_j \geq b \}$.
	If $0 < \sum_{j \in [n]\backslash\CS} c_j < b$, then the \MIR inequality \eqref{CMIRineq} with $\L =\CS$, $\U = [n]\backslash\CS$, and $\delta=1$ reduces to
	\begin{equation}
		\label{coverineq-1}
		\sum_{j \in  \CS} x_j \geq 1,
	\end{equation}
	which is stronger than the strengthened Benders feasibility cut \eqref{sBDcut}.
\end{proposition}
\begin{proof}
	Suppose that $\L = \CS$, $\U = [n]\backslash \CS$, $\delta=1$.
	Then from \eqref{gdef}, \cref{abobser}, and $\sum_{j \in [n]\backslash\CS} c_j < b$, we have $\lceil\beta \rceil= \lceil b - \sum_{j \in \setN\backslash\CS}c_j \rceil =1$ and $f_\beta = b-\sum_{j \in \setN\backslash\CS} c_j$.
	 For $j \in \CS$, if $c_j = 1$, it follows $G(c_j)= \lfloor c_j \rfloor + \min \left\{\frac{f_{c_j}}{f_\beta},1\right\}= 1 $; otherwise, 
	  it also follows $f_{c_j}=c_j$ and $G(c_j)= \lfloor c_j \rfloor + \min \left\{\frac{f_{c_j}}{f_\beta},1\right\}=  \min \left\{\frac{c_j}{b - \sum_{j' \in [n]\backslash\CS}c_{j'}},1\right\}=1$.
	For $j \in [n]\backslash\CS$, if $c_j=0$, it follows $G(-c_j) = \lfloor -c_j \rfloor+\min 
	\left\{\frac{f_{-c_j}}{f_\beta},1\right\}=0$; otherwise, 
	 $f_{-c_j}=1-c_j >0$ and $G(-c_j) = \lfloor -c_j \rfloor+\min 
	\left\{\frac{f_{-c_j}}{f_\beta},1\right\}=-1 +\min \left\{\frac{1-c_j}{b - \sum_{j' \in \setN\backslash\CS}c_{j'}},1\right\}=0$.
	Thus, \eqref{CMIRineq} reduces to \eqref{coverineq-1}.
	By $\sum_{j\in \setN} \min \{c_j, b\}x_j = b \sum_{j \in  \CS} x_j +  \sum_{j \in [n]\backslash\CS} c_jx_j  \geq b \sum_{j \in \CS} x_j$,  \eqref{coverineq-1} must be stronger than \eqref{sBDcut}.
\end{proof}
\cref{pro:eqmir,pro:coef-1} demonstrate that by choosing appropriate $(\L,\U)$ and $\delta$, we can obtain an \MIR-enhanced Benders feasibility cut \eqref{CMIRineq} that is either identical to or stronger than the strengthened Benders feasibility cut \eqref{sBDcut} (under certain conditions).
This shows the potential of the \MIR-enhanced Benders feasibility cut in strengthening the \LP relaxation of formulation \eqref{mp-ipscp}.
In \cref{subsec:techniques}, we will further demonstrate this by numerical experiments.

	\section{Numerical results}
\label{sect:numer}
In this section, we present the computational results to show the effectiveness and efficiency of the proposed \BD algorithm and enhancement techniques for the \PSCP.
To do this, we first perform experiments on small-scale \PSCP instances to demonstrate the advantage of the proposed \BD algorithm over state-of-the-art \MIP solvers.
Then, we present detailed computational results of the proposed \BD algorithm on large-scale \PSCP instances.
Finally, we evaluate the performance impact of the  enhancement techniques developed in \cref{sect:implementation} for the \PSCP \footnote{\rev{In Section 2 of the online supplement, we present the computational results to compare the proposed \BD algorithm with the state-of-the-art approach in \cite{Jiang2022} on the \PSCP with $m=1$.}}.

The proposed \BD algorithm was implemented in Julia 1.7.3 using \CPLEX 20.1.0. 
We set parameters of \CPLEX to run the code in a single-threaded mode, with a time limit of 7200 seconds and a relative \MIP gap tolerance of 0\%.
Unless otherwise specified, other parameters in \CPLEX were set to their default values.
All computational experiments were performed on a cluster of Intel(R) Xeon(R) Gold 6140 \CPU{} @ 2.30GHz computers.

\subsection{Testsets}\label{subsect:testsets}
In our experiments, we construct the \PSCP instances using the $60$ deterministic \SCP instances with up to $500$ rows and $5000$ columns, as considered in  \cite{Fischetti2012}.
These instances are publicly available at the ORLIB library \footnote{\url{http://people.brunel.ac.uk/~mastjjb/jeb/orlib/scpinfo.html}}.
We use two different distributions of random vectors $A_i$ to construct the scenarios (vectors) $A_i^\omega$ for the \PSCP instances. 
In the first case, each $a_{ij}$ is assumed to be an independent Bernoulli random variable \citep{Hwang2004,Fischetti2012}.
Specifically, for a constraint matrix $A$ of a deterministic \SCP instance, each $a_{ij}$, $j \in \supp(A_i)$, has a probability $p_{ij}$ to be disappeared, randomly chosen from $[0, 0.4]$, and each $j \in [n] \backslash \supp(A_i)$ has a probability $1$ to be disappeared.
In the second case, each row $(a_{i1}, \ldots, a_{in})$ is assumed to be conditionally independent, following the Bernoulli mixture distribution considered in \cite{Ahmed2013} (and thus random variables $a_{ij_1}$ and $a_{ij_2}$ can be {correlated}).
In particular, given a finite prior distribution $\{\pi_1, \ldots, \pi_{L}\}$, the conditional probabilities $p_{ij_1\ell}$ and $p_{ij_2\ell}$ (corresponding to the disappearance of columns $j_1$ and $j_2$ in scenario $\ell$) are independent for $j_1,j_2 \in [n]$ with $j_1 \neq j_2$ and $\ell \in [L]$.
In our experiments, $L$ is set to $50$; each $\pi_\ell$ is uniformly chosen from {$[0, 1]$} and normalized such that $\sum_{\ell=1}^L \pi_{\ell}=1$; and similar to the independent case, $p_{ij\ell}$, $j \in \supp(A_i)$, are uniformly chosen from $[0, 0.4]$, and  $p_{ij\ell}$, $j \in [n] \backslash \supp(A_i)$, are set to $1$.

We construct the \PSCP  instances based on the above two independent and correlated distributions of random vectors $A_i$, and the instances are thus referred as to independent and correlated instances, respectively.
The first testset \Tone consists of small-scale instances in which $s$ is set to $100$ and $\epsilon_i$, $i \in [m]$, are set to the common value $\epsilon$, chosen from $\{0.05, 0.1\}$.
The second testset \Ttwo consists of large-scale instances in which $s$ is chosen from $\{1000, 2000\}$ and similarly, $\epsilon_i$, $i \in [m]$, are set to the common value $\epsilon$, chosen from $\{0.025,0.05, 0.1\}$.
In total, there are $240$ and $720$ instances in testsets \Tone and \Ttwo, respectively.
\rev{It deserves to mention that in all constructed instances, condition \eqref{infeascheck} does not hold for all $i \in [m]$, and thus all constructed instances are feasible.}

\subsection{Comparison with state-of-the-art \MIP solvers}
In this subsection, we show the advantage of the proposed \BD algorithm for the \PSCP based on formulation \eqref{mp-ipscp} over the direct use of \CPLEX applied as a black-box \MIP solver to formulation \eqref{IP-IPSCP}.
To do this, we compare the following settings:
\begin{itemize}
	\item \CPX: formulation \eqref{IP-IPSCP} is solved using \CPLEX's  branch-and-cut solver;
	\item  \AUTO: formulation \eqref{IP-IPSCP} is solved using \CPLEX's automatic \BD algorithm (by setting parameter "Benders strategy" to "full");
	\item \SBD:  formulation \eqref{mp-ipscp} is solved using the proposed \BD algorithm in which the {Benders feasibility cuts} are separated for integer and fractional solutions at all nodes of the search tree;
	\item \RBD:  formulation \eqref{mp-ipscp} is solved using the proposed \BD algorithm {in which the {Benders feasibility cuts} are separated for integer solutions at all nodes and fractional solutions at the root node of the search tree}.
\end{itemize}
Notice that (i) in settings \CPX and \AUTO, the integrality constraints \eqref{ip_ipscp_cons3} in formulation \eqref{IP-IPSCP} are (equivalently) relaxed into $0 \leq z_i^\omega \leq 1$ for $i \in  \setM$ and $\omega \in \setS$; and (ii) in settings \SBD and \RBD, the three  enhancement techniques (i.e., initial cuts, the customized \RENS heuristic algorithm, and the \MIR-enhanced Benders feasibility cuts) in \cref{sect:implementation} are all implemented. 

\cref{fig1a,fig1b} give the performance profiles comparing the CPU times and the end gaps (when the time limit was hit) returned by settings \CPX, \AUTO, \SBD, and \RBD.
Detailed statistics of instance-wise computational results can be found in \cref{tab:table1,tab:table2,tab:table3,tab:table4} in the Appendix.
First, from \cref{fig1a,fig1b}, we can conclude that the performances of the proposed \SBD and \RBD are fairly comparable on the small-scale instances in testset \Tone. 
Overall, \RBD seems to perform slightly better than \SBD on easy instances.
Second, we observe from the two figures that the proposed \SBD and \RBD outperform \CPX by at least one order of magnitude. 
In particular, as demonstrated in \cref{fig1a}, almost all instances can be solved by \SBD and \RBD within the time limit, while less than 50\% of the instances can be solved by \CPX.
For the unsolved instances, the end gap returned by \CPX is also much larger than those returned by \SBD and \RBD, as shown in \cref{fig1b}. 
Finally, we observe that \AUTO also performs fairly well.
In particular, \AUTO can also solve most instances to optimality within the time limit.
This may be explained by the reason that 
the main difference between our proposed \SBD/\RBD and \AUTO lies in the way to solve the Benders subproblem \eqref{sp-ipscp-1} to obtain the {Benders feasibility cut} \eqref{inifeacut_ipscp}, and according to \cite{Bonami2020}, the Benders subproblem \eqref{sp-ipscp-1} may also be efficiently solved by \AUTO with the technique to handle the so-called \emph{generalized bound constraints} \eqref{ip_ipscp_cons1}.
Nevertheless, we observe that \AUTO is outperformed by the proposed \SBD and \RBD, especially for the easy instances. 
In particular, more than $50\%$ of the instances can be solved within $10$ seconds by the proposed \SBD and \RBD while only about $35\%$ of the instances can be solved within 10 seconds by \AUTO.

\begin{figure}[t]
		\centering 
		\subfloat[]{\label{fig1a}\includegraphics[width=0.47\textwidth]{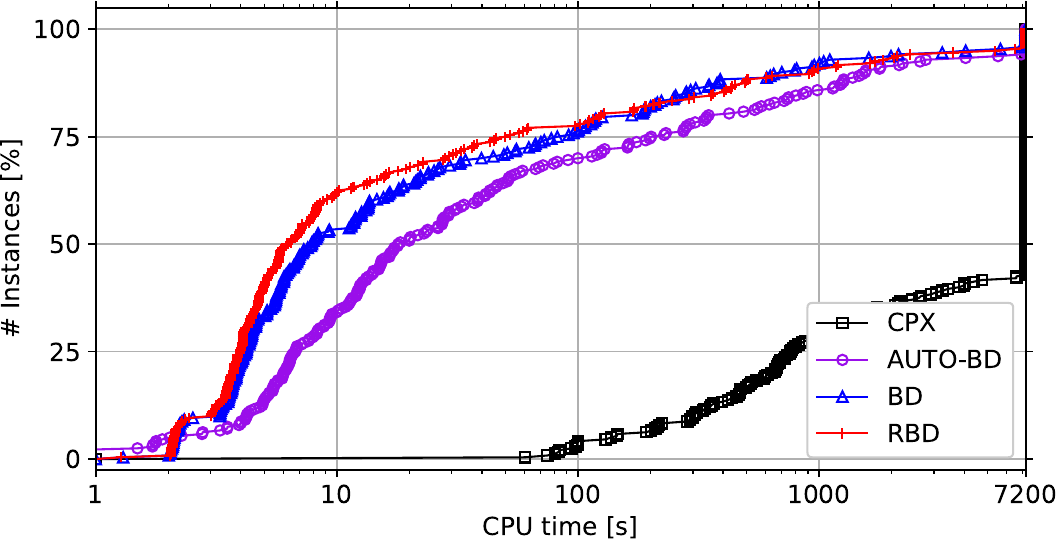}}
		\qquad
		\subfloat[]{\label{fig1b}\includegraphics[width=0.47\textwidth]{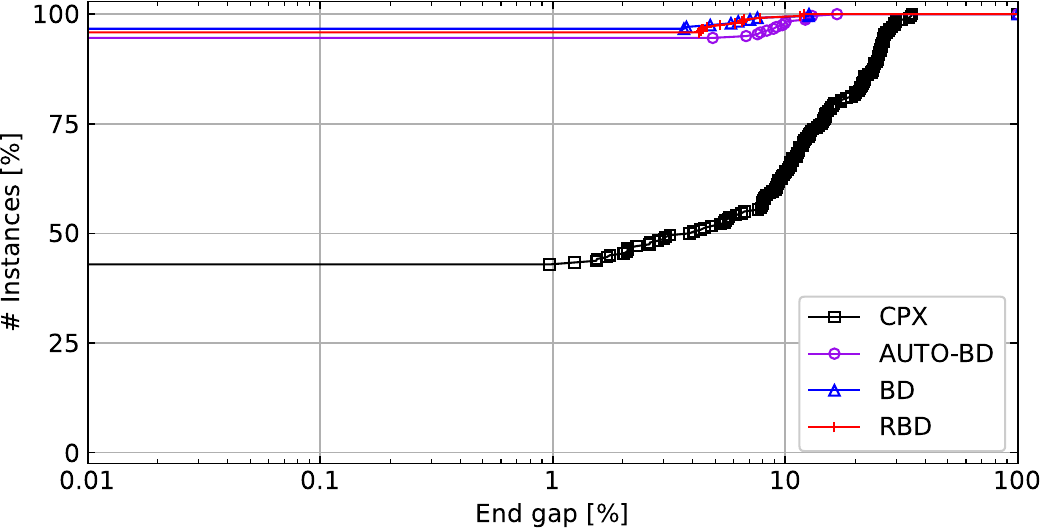}}
		\caption{Performance profiles of the CPU times and end gaps on the instances in testset \Tone.}
		\label{fig1}
\end{figure}

\subsection{Computational results on large-scale instances}
We now evaluate the performances of the proposed \SBD and \RBD on the large-scale instances in testset \Ttwo. 
\cref{fig2} summarizes the computational results, grouped by $\epsilon$, $s$, and the type of instances (i.e., independent or correlated instances).
Detailed statistics of instance-wise computational results can be found in Tables \rev{2-13} of the online supplement.

\begin{figure}[h!]
		\centering 
		\subfloat[]{\label{fig2a}\includegraphics[width=0.47\textwidth]{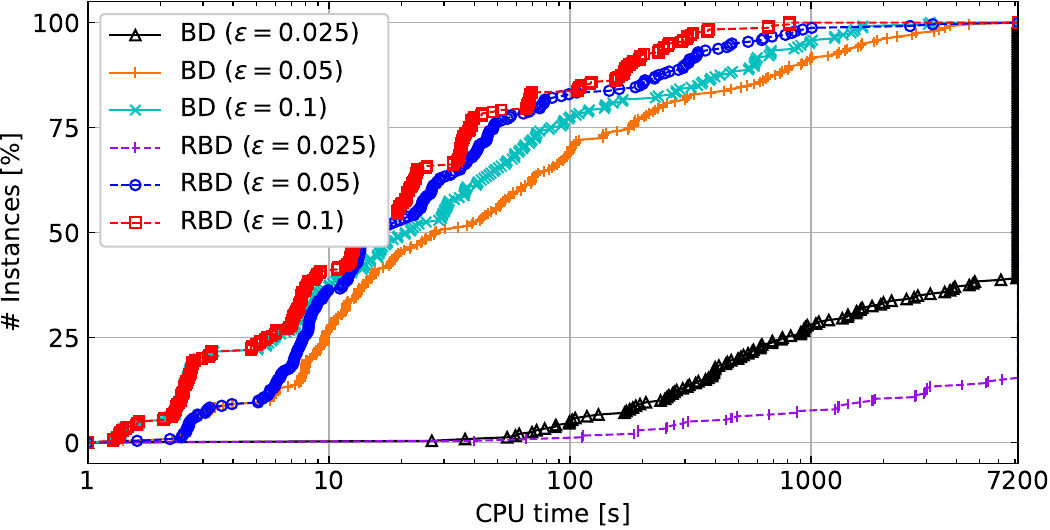}}
		\qquad
		\subfloat[]{\label{fig2b}\includegraphics[width=0.47\textwidth]{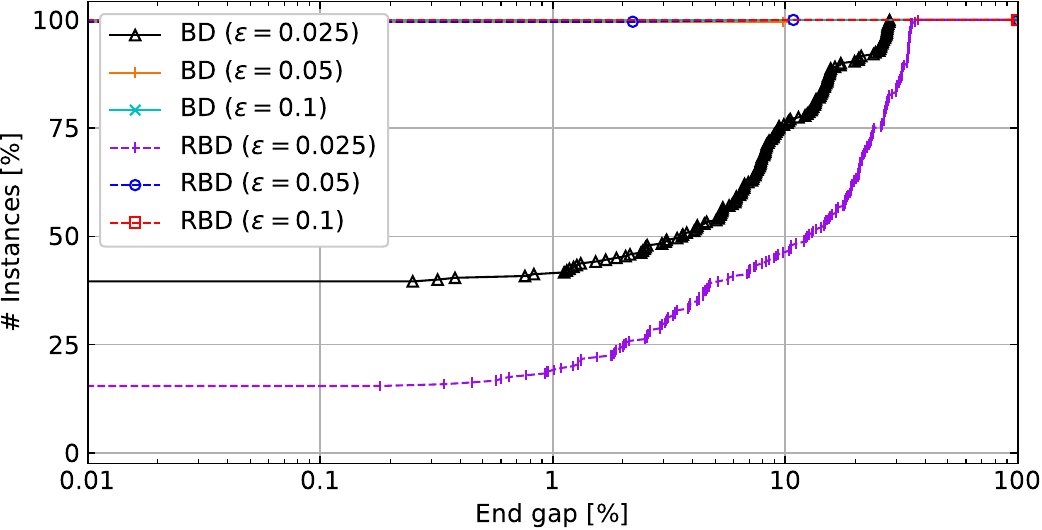}}
		\qquad
		\subfloat[]{\label{fig3a}\includegraphics[width=0.47\textwidth]{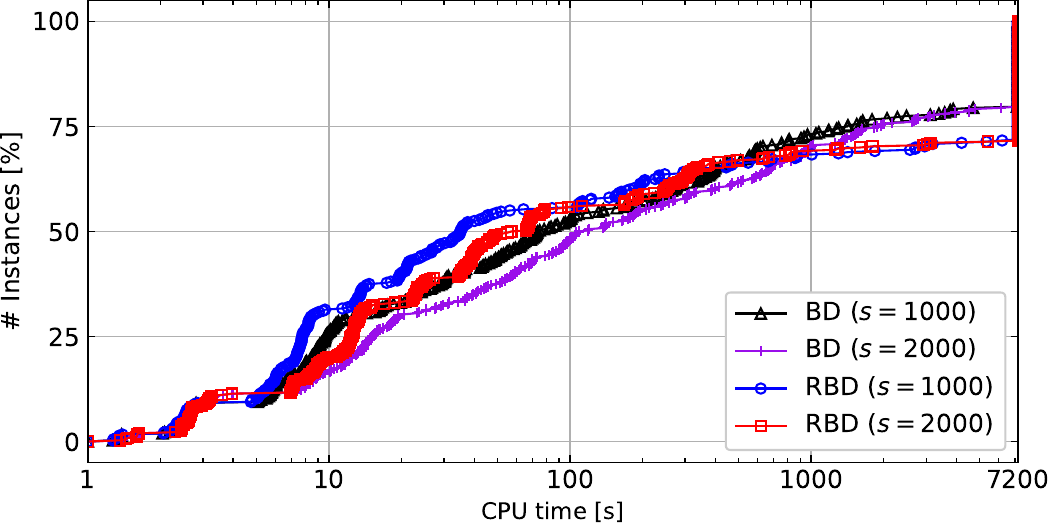}}
		\qquad
		\subfloat[]{\label{fig3b}\includegraphics[width=0.47\textwidth]{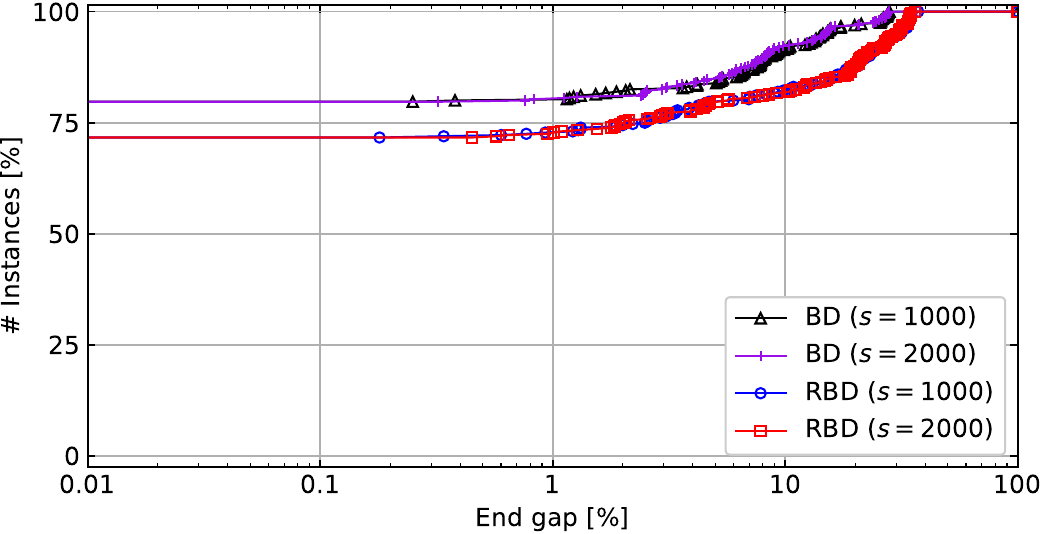}}
		\qquad
		\subfloat[]{\label{fig4a}\includegraphics[width=0.47\textwidth]{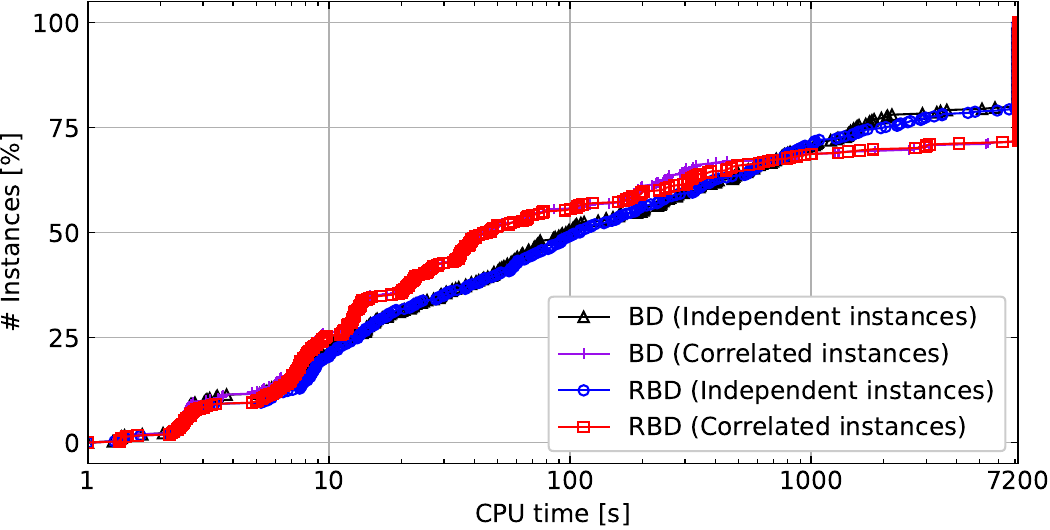}}
		\qquad
		\subfloat[]{\label{fig4b}\includegraphics[width=0.47\textwidth]{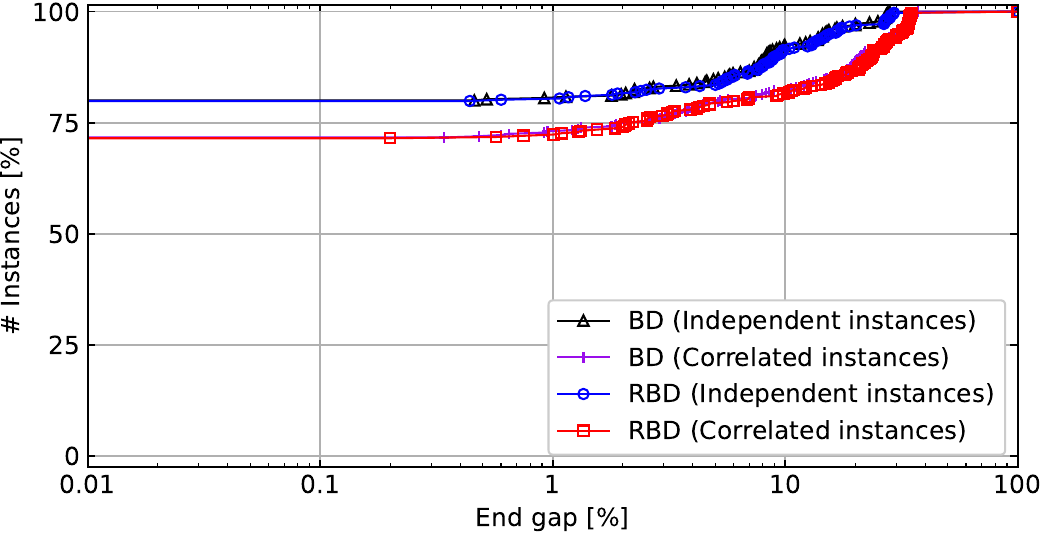}}
		\vspace*{-0.4cm}
		\caption{Performance profiles of the \CPU times and end gaps on the instances in testset \Ttwo.
		 (a) and (b): grouped by $\epsilon$; (c) and (d): grouped by $s$; (e) and (f): grouped by independent and correlated instances.}
	 	\vspace*{-0.4cm}
		\label{fig2}
\end{figure}

\begin{figure}[t]
	\centering
	\subfloat[]{\label{rootgapepsilon}\includegraphics[width=0.47\textwidth]{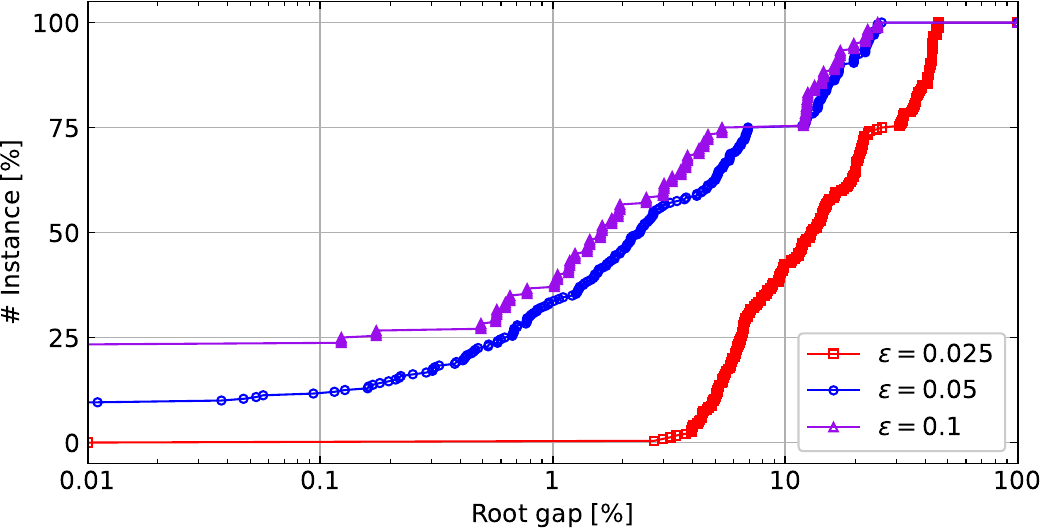}}
	\qquad
	\subfloat[]{\label{figupper}\includegraphics[width=0.47\textwidth]{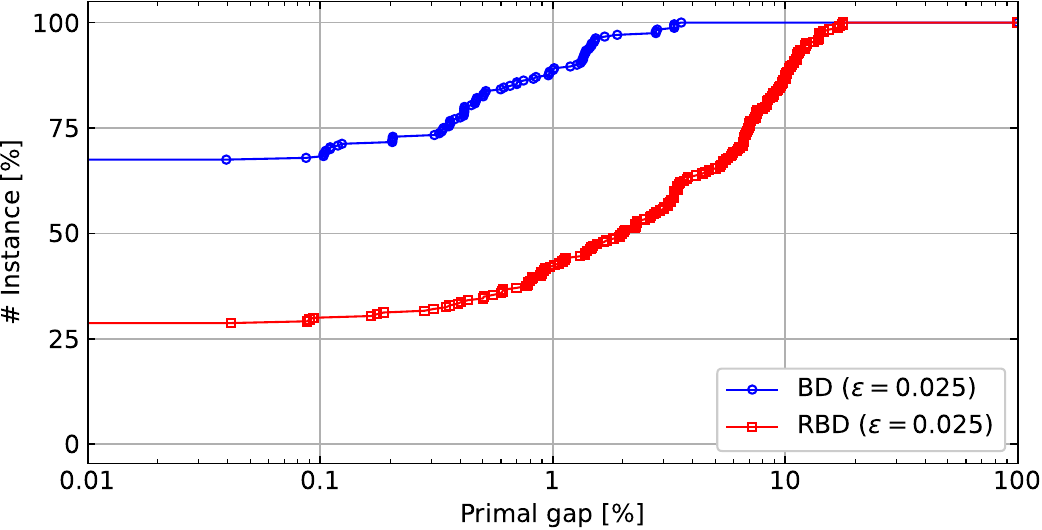}}
	\vspace*{-0.4cm}
	\caption{(a) Performance profiles of the root gaps on the instances in testset \Ttwo; (b) Performance profiles of the primal gaps on the instances with $\epsilon=0.025$ in testset \Ttwo.}
	\vspace*{-0.4cm}
\end{figure}

First, we can observe from \cref{fig2a,fig2b} that the \PSCP instances with $\epsilon=0.025$ are much more difficult to be solved by \SBD and \RBD compared with those with $\epsilon = 0.05, 0.1$.
Specifically, for instances with $\epsilon = 0.05, 0.1$, \RBD and \SBD can solve all of them to optimality while for instances with $\epsilon = 0.025$, \RBD and \SBD were only able to solve about $15\%$ and {$35\%$} of the instances to optimality, respectively.
The reason can be explained as follows. 
The \PSCP instances with $\epsilon=0.025$ are much more constrained than those with $\epsilon = 0.05, 0.1$, which not only {leads to a weak \LP bound} at the root node but also {makes \SBD and \RBD challenging to find optimal or near-optimal solutions}.
To justify this, we present the performance profiles of the root gaps of the \LP bound at the root node and the primal gaps of the solution (for instances with $\epsilon=0.025$) returned by \SBD and \RBD in \cref{rootgapepsilon,figupper}, respectively. 
The root gap and primal gap of an instance are computed  as $100 \times \frac{o^*-o_{\text{root}}}{o^*}$ and $100 \times \frac{o^F - o^*}{o^*}$, respectively, where 
$o^*$, $o_{\text{root}}$, and $o^{F}$ are the optimal value or the best incumbent \footnote{For an unsolved instance by \SBD and \RBD with a 7200 seconds time limit, the best incumbent $o^{*}$ is obtained by solving it using \SBD with a 36000 seconds time limit. Note that \SBD performs better than \RBD on hard instances with $\epsilon=0.025$, as shown in \cref{fig2b}.}, the LP bound obtained at the root node, and the objective value of the feasible solution returned by \SBD/\RBD, respectively.
As shown in \cref{rootgapepsilon,figupper}, (i) the root gaps of instances with $\epsilon=0.025$ are much larger than those of instances with $\epsilon=0.05,0.1$; and (ii) the primal gaps of instances with  $\epsilon=0.025$ are \rev{fairly large}, which confirms our statement.
It is worthy noting that the different quality of the \LP bounds also makes the different performance behavior of the proposed \SBD and \RBD on instances with different $\epsilon$.
Overall, we observe that for instances with $\epsilon = 0.05,0.1$, \RBD performs better than \SBD, while for those with $\epsilon=0.025$, \SBD performs much better.
This is reasonable as instances with $\epsilon=0.025$ generally associate with a weak \LP bound (as shown in \cref{rootgapepsilon}), and the strategy of adding cuts for fractional solutions encountered in the whole search tree in \SBD can strengthen the lower bound more quickly, and thus makes the performance of \SBD better than that of \RBD.
For instances with $\epsilon=0.05, 0.1$ that generally associate with a relatively good \LP bound, the same strategy cannot compensate for the additional overhead of making a large \LP relaxation (as more cuts are added), and thus makes the performance of \SBD worse than that of \RBD.

Next, we compare the results of the \PSCP instances with different numbers of scenarios $s$. 
\cref{fig3a,fig3b} show that solving instances with $s=2000$ by \SBD and \RBD is only slightly more difficult  than solving those with $s=1000$.
This is due to the reasons that (i) for the \PSCP instances with different $s$, the numbers of variables in \eqref{mp-ipscp} are identical; and (ii) for the proposed \SBD or \RBD, the main difference in solving the \PSCP instances with different $s$ lies in how fast we obtain the Benders feasibility cuts \eqref{feas-cut2}, and
as demonstrated in \cref{closed_formula}, the {Benders feasibility cuts} \eqref{feas-cut2} can be computed by an efficient polynomial-time algorithm.
Indeed, in our experiments, we observed that the overhead spent in the computations of the Benders feasibility cuts \eqref{feas-cut2} is not large, especially for the \RBD.
These results highlight the scalability of the proposed \SBD and \RBD in solving large-scale \PSCP instances (with large numbers of scenarios).

Finally, we observe from \cref{fig4a,fig4b} that the performances of \SBD and \RBD on the independent and correlated \PSCP instances are quite similar.
This is reasonable knowing that (i) our proposed \SBD and \RBD for the \PSCP do not require a specific distribution of the random data $A_i$; and (ii) the independent and correlated distributions of $A_i$ that are used to construct the scenarios $A_i^\omega$ are quite similar (indeed, the independent distribution of $A_i$ can be seen as a special case of the correlated distribution of $A_i$ with $L=1$).

\subsection{Performance effects of the three enhancement techniques in \cref{sect:implementation}}
\label{subsec:techniques}

In this subsection, we evaluate the performance effects of the three enhancement techniques (i.e., initial cuts, the customized \RENS heuristic algorithm, and the \MIR-enhanced Benders feasibility cuts)
in \cref{sect:implementation} for solving the \PSCP.
To do this, we compare the following variants of setting \SBD:
\begin{itemize}
	\item \BasicSBD: the basic version of setting \SBD  in which none of the three techniques is implemented;
	\item \BasicSBDI: the setting \BasicSBD together with the initial cuts proposed in \cref{subsec:initialcuts};
	\item \BasicSBDIH: the setting \BasicSBDI together with the customized \RENS heuristic algorithm proposed in \cref{subsec:primalheuristic}; 
	\item  \BasicSBDIHV: the setting \BasicSBDIH together with the \MIR-enhanced Benders feasibility cuts proposed in \cref{subsec:validineqs}, i.e., setting \SBD.
\end{itemize}
We do not compare the corresponding variants of \RBD as our preliminary experiments showed that the obtained results are similar to those presented in this subsection.
\cref{fig5} displays the performance profiles of the \CPU times, end gaps, root gap\roundtwo{s}, and number of explored nodes on the instances in testset \Ttwo returned by the four settings.
Detailed statistics of instance-wise computational results can be found in Tables \rev{14-25} of the online supplement.
\begin{figure}[t]
		\centering 
		\subfloat[]{\label{fig5a}\includegraphics[width=0.47\textwidth]{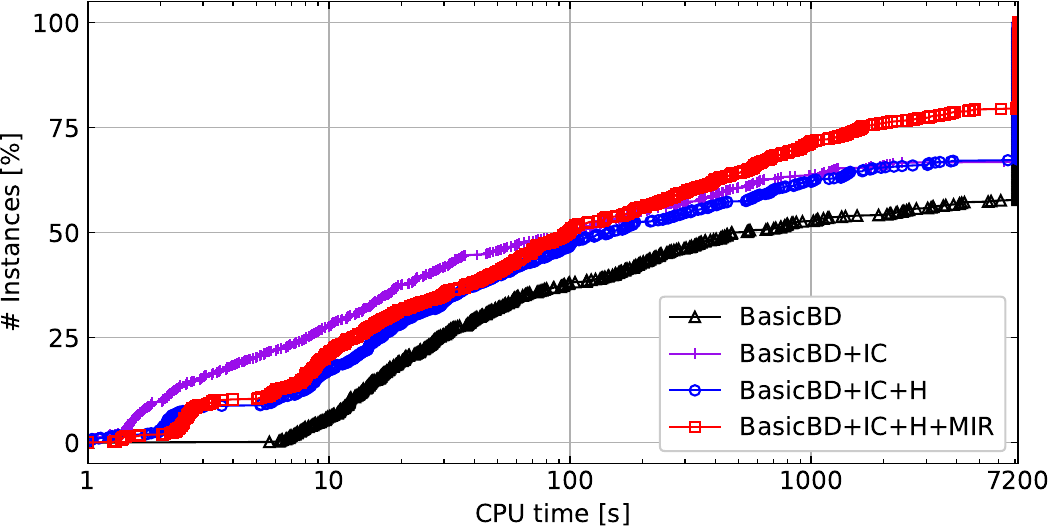}}
		\quad
		\subfloat[]{\label{fig5b}\includegraphics[width=0.47\textwidth]{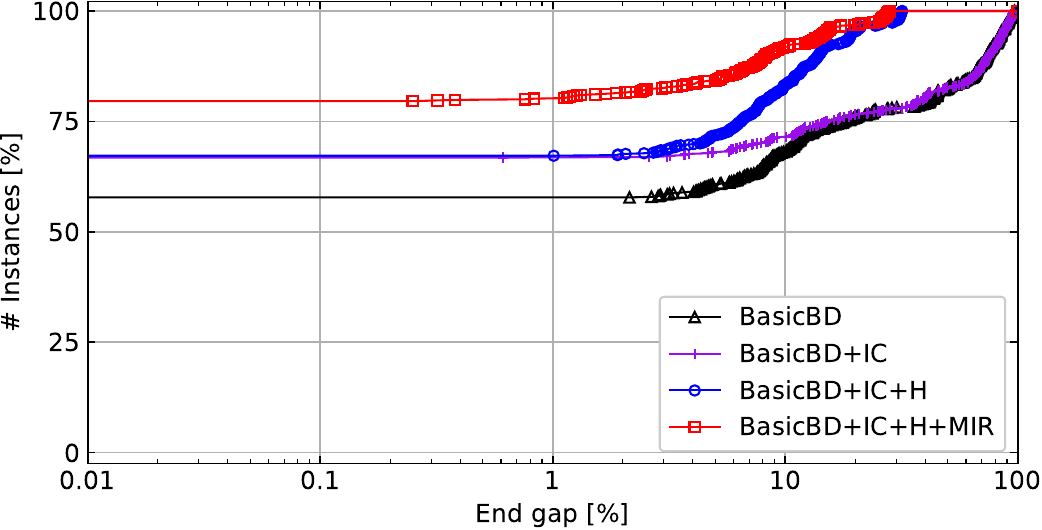}}
		\quad
		\subfloat[]{\label{InitialCuts}\includegraphics[width=0.47\textwidth]{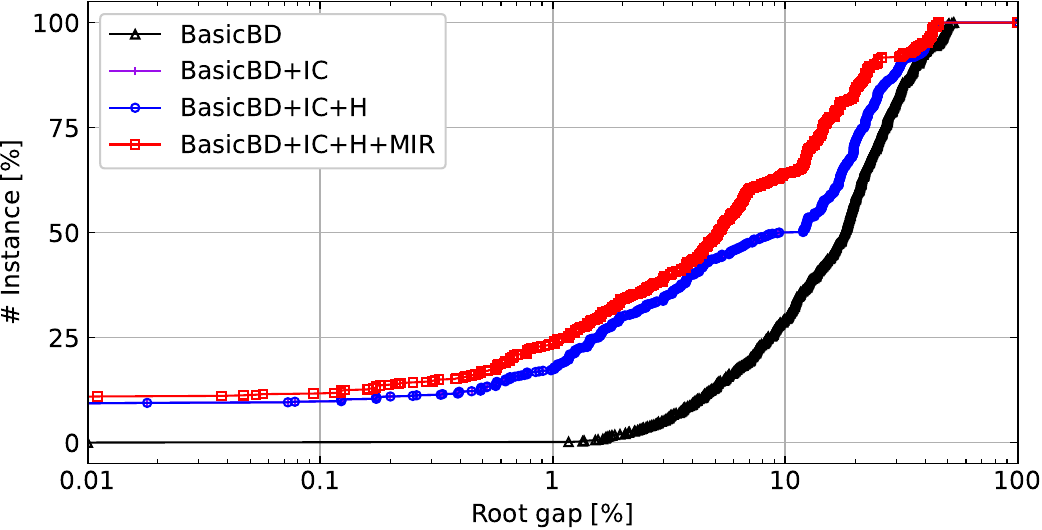}}
		\quad
		\subfloat[]{\label{ICNode}\includegraphics[width=0.47\textwidth]{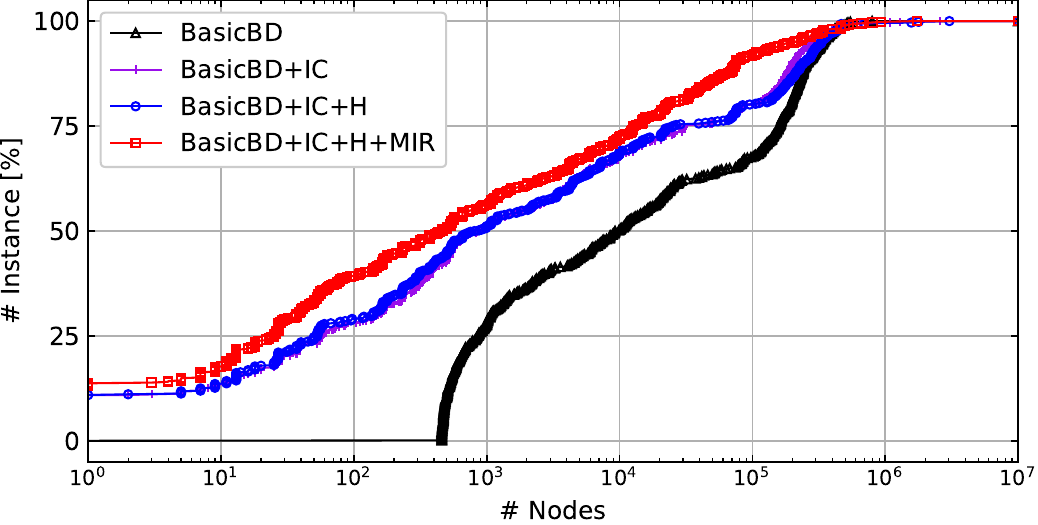}}
		\caption {Performance profiles of the \CPU times, end gaps, root gap\roundtwo{s}, and number of explored nodes on the instances in testset \Ttwo.}
		\label{fig5}
\end{figure}

First, we can observe from \cref{fig5a} that compared with those returned by \BasicSBD, the number of solved instances returned by \BasicSBDI is much larger and the CPU time returned by \BasicSBDI is much smaller.
The reason can be seen in \cref{InitialCuts,ICNode} in which we observe that with the initial cuts \eqref{initcut-pre}, \BasicSBDI can return a much better root gap and a much smaller number of explored nodes than \BasicSBD. 
This shows that initial cuts  \eqref{initcut-pre} can indeed strengthen the \LP relaxation of formulation \eqref{mp-ipscp} (as they enable \CPLEX to construct further  internal cuts)
and thus improve the overall performance of the proposed \BD algorithm.

\begin{figure}[h]
		\centering
		\includegraphics[width=0.47\textwidth]{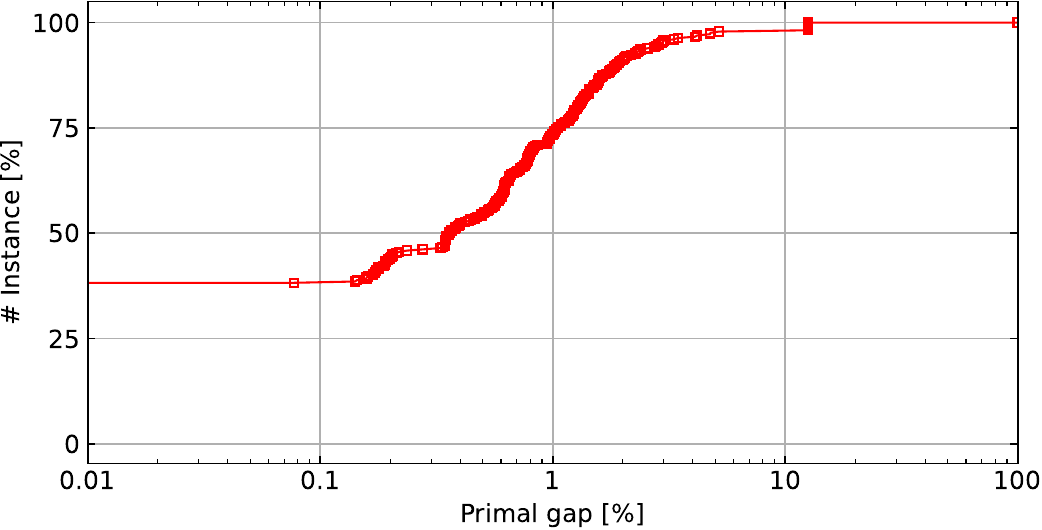}
			\vspace*{-0.3cm}
		\caption {Performance profile of the primal gaps of the solutions returned by \RENS on the instances in testset \Ttwo.}
		\label{Heur}
		\vspace*{-0.3cm}
\end{figure}

Next, we evaluate the performance effect of the customized \RENS heuristic algorithm in \cref{subsec:primalheuristic}.
Comparing \BasicSBDI and \BasicSBDIH in \cref{fig5a,fig5b}, 
we find that although the customized \RENS heuristic algorithm does not contribute to increase the number of solved instances and reduce the overall solution time of the solved instances,
it can effectively reduce the end gap of the unsolved instances.
To gain deeper insight into the computational efficiency of the customized \RENS heuristic algorithm,
we report the {primal gaps} of the solutions returned by the customized \RENS heuristic algorithm in \cref{Heur}.
As shown in \cref{Heur}, for about 75\% of the instances, the primal gap of the solutions returned by \BasicSBDIH is less than $1\%$, 
which highlights the effectiveness of the customized \RENS heuristic algorithm in providing a high-quality solution for the \PSCP and improving the overall performance of the \BD algorithm on the hard instances.

Finally, we evaluate the performance effect of the \MIR-enhanced Benders feasibility cuts in \cref{subsec:validineqs}.
As shown in \cref{InitialCuts}, with the \MIR-enhanced Benders feasibility cuts, the root gap returned by \BasicSBDIHV is much smaller than that returned by \BasicSBDIH, which shows the effectiveness of the \MIR-enhanced Benders feasibility cuts in strengthening the \LP relaxation of formulation \eqref{mp-ipscp}. 
As a result, equipped with the \MIR-enhanced Benders feasibility cuts, the proposed \BD algorithm performs much better. 
In particular, the number of explored nodes and the CPU time (of the hard instances) returned by \BasicSBDIHV are much smaller, and more than 10\% instances can be solved to optimality with the time limit.

In summary, our results show that the three enhancement techniques presented in \cref{sect:implementation} can all improve the overall performance of the \BD algorithm in terms of achieving a better solution time or a better end gap.

\section{Conclusions and future works}
\label{sect:conclusion}

In this paper, we consider the \PSCP in which each row of the constraint matrix is a Bernoulli random vector and has a finite discrete distribution. 
We develop an efficient \BD algorithm based on a Benders reformulation of the \PSCP.
Two key features of the proposed \BD algorithm, which make it particularly suitable to solve large-scale \PSCP{s}, are: 
(i) the number of variables in the Benders reformulation is equal to the number of columns but independent of the number of scenarios of the random rows; and
(ii) the Benders feasibility cuts can be separated by efficient polynomial-time algorithms. 
We have also proposed three enhancement techniques (i.e., initial cuts, the customized \RENS heuristic algorithm, and the \MIR-enhanced Benders feasibility cuts) to improve the performance of the proposed \BD algorithm. 
Extensive computational results showed that the proposed \BD algorithm outperforms the state-of-the-art general-purpose \MIP solver's branch-and-cut and automatic \BD algorithms.
Furthermore, the proposed \BD algorithm is capable of solving large-scale \PSCP instances with up to $500$ rows, $5000$ columns, and $2000$ scenarios of the random rows.

As for future work, we will develop more effective cutting planes to improve the performance of the proposed \BD algorithm for solving the \PSCP, especially for the case where $A_i x \geq 1$, $i \in [m]$, are required to be satisfied at a high probability $1-\epsilon$ (e.g., $\epsilon=0.025$).
 \rev{In addition, it would also be interesting to extend the proposed \BD algorithm to solve the joint \PSCP: $\min ~ \{ c^\top x \,:\, \mathbb{P}\{Ax \geq \boldsymbol{1} \}\geq1-\epsilon,~x \in \{0,1\}^n \}$ \citep{Beraldi2010}. 
\roundtwo{One difficulty of this extension is that the big-M reformulation of the joint \PSCP provides a very weak \LP relaxation bound, hindering the
convergence of the \BD algorithm; see Section 4 of the online supplement for more details.}
In order to develop an efficient \BD algorithm for solving the joint \PSCP, we need  to develop effective cutting planes (by judiciously exploiting the special structure of the problem) to strengthen the big-M reformulation of the joint \PSCP.}\\[5pt]
{\noindent\bf Acknowledgement }\\[5pt]
\rev{We would like to thank the two anonymous reviewers for their insightful comments.
{Jie Liang, Cheng-Yang Yu, and Wei Lv contributed equally and are joint first authors.}
The work of Wei-Kun Chen was supported in part by the Chinese NSF grants (No. 12101048).
The work of Yu-Hong Dai was supported in part by the Chinese NSF grants (Nos. 12021001 and 11991021).}

\newpage
\section*{Appendix. Detailed computational results}

{
	\centering
	\mytabcolsepeach
	\myarraystretch
	\centering
	\scriptsize
	\begin{longtable}{rrrrrrrrrr}
		\caption{Performance comparison of settings \CPX, \AUTO, \SBD, and \RBD for the independent instances with $\epsilon =0.05$ in testset \Tone.             \T (\G \%) denotes that the CPU time is \T if the instance is solved within the timelimit; otherwise, it denotes that the end gap returned by \CPLEX is \G \%. $\NODE$ denotes the number of explored nodes.}\\
		\label{tab:table1}\\
		\headlinea
		\multirow{2}{*}{\texttt{Name}} & \multicolumn{2}{c}{\CPX} & \multicolumn{2}{c}{\AUTO} & \multicolumn{2}{c}{\SBD} & \multicolumn{2}{c}{\RBD}\\\cmidrule(r){2-3}\cmidrule(r){4-5}\cmidrule(r){6-7} \cmidrule(r){8-9} & {\T (\G \%)} & \multicolumn{1}{r}{\NODE} & {\T (\G \%)} & \multicolumn{1}{r}{\NODE} & \T (\G \%) & \multicolumn{1}{r}{\NODE} & \T (\G \%) & \multicolumn{1}{r}{\NODE}\\
		\headlineb
		\endhead
		\footlinea
		\multicolumn{9}{r}{continued on next page}\\
		\footlineb
		\endfoot
		\endline
		\endlastfoot
		scp41 & 516.0 & 1298 & 6.2 & 236 & 4.5 & 22 & 4.2 & 86 \\ 
		scp42 & 675.3 & 1563 & 8.3 & 542 & 4.3 & 53 & 4.3 & 464 \\ 
		scp43 & 411.7 & 1163 & 6.9 & 595 & 4.2 & 58 & 4.1 & 155 \\ 
		scp44 & 60.3 & 11 & 4.5 & 0 & 3.6 & 5 & 3.5 & 4 \\ 
		scp45 & 208.6 & 581 & 4.9 & 35 & 3.5 & 5 & 3.5 & 5 \\ 
		scp46 & 82.5 & 35 & 6.4 & 18 & 3.8 & 7 & 3.5 & 7 \\ 
		scp47 & 94.6 & 120 & 5.6 & 84 & 3.6 & 7 & 3.6 & 11 \\ 
		scp48 & 95.8 & 47 & 5.2 & 83 & 3.7 & 9 & 3.7 & 12 \\ 
		scp49 & 1191.5 & 3401 & 11.2 & 2050 & 6.0 & 312 & 5.4 & 1263 \\ 
		scp410 & 288.3 & 668 & 7.9 & 378 & 4.5 & 46 & 4.4 & 251 \\ 
		scp51 & 365.2 & 805 & 12.1 & 715 & 5.4 & 97 & 4.5 & 280 \\ 
		scp52 & 1389.3 & 4284 & 13.5 & 801 & 5.7 & 82 & 4.8 & 311 \\ 
		scp53 & 1306.2 & 3955 & 13.9 & 643 & 6.0 & 115 & 5.1 & 597 \\ 
		scp54 & 1342.3 & 4891 & 14.9 & 1340 & 6.0 & 103 & 4.9 & 454 \\ 
		scp55 & 72.5 & 51 & 8.1 & 0 & 3.8 & 7 & 3.8 & 9 \\ 
		scp56 & 150.1 & 165 & 9.6 & 54 & 4.2 & 13 & 4.2 & 23 \\ 
		scp57 & 354.8 & 881 & 12.0 & 186 & 5.0 & 52 & 4.5 & 137 \\ 
		scp58 & 1216.5 & 5037 & 12.5 & 1224 & 6.7 & 236 & 5.0 & 614 \\ 
		scp59 & 223.2 & 329 & 13.0 & 159 & 4.7 & 35 & 4.3 & 70 \\ 
		scp510 & 2034.8 & 6250 & 14.2 & 921 & 5.8 & 125 & 5.6 & 987 \\ 
		scp61 & 4401.4 & 19537 & 11.1 & 2240 & 6.7 & 457 & 7.0 & 1823 \\ 
		scp62 & 1552.1 & 6165 & 8.6 & 1868 & 5.7 & 195 & 5.9 & 1351 \\ 
		scp63 & 397.5 & 1446 & 5.2 & 667 & 4.5 & 53 & 4.1 & 186 \\ 
		scp64 & 311.3 & 897 & 6.7 & 314 & 4.6 & 37 & 4.2 & 156 \\ 
		scp65 & (5.2) & 13594 & 68.7 & 16238 & 23.2 & 4308 & 36.5 & 21956 \\ 
		scpa1 & (4.5) & 4800 & 124.8 & 10244 & 43.3 & 3118 & 32.3 & 8631 \\ 
		scpa2 & (1.5) & 7649 & 114.9 & 6831 & 14.4 & 434 & 16.4 & 4191 \\ 
		scpa3 & (2.1) & 6229 & 38.7 & 1683 & 8.3 & 127 & 8.3 & 1605 \\ 
		scpa4 & (3.0) & 5964 & 45.5 & 2802 & 11.8 & 529 & 11.8 & 2197 \\ 
		scpa5 & (2.0) & 5710 & 51.4 & 3960 & 11.4 & 352 & 9.7 & 1743 \\ 
		scpb1 & (11.4) & 5594 & 335.2 & 31402 & 87.5 & 12834 & 122.3 & 57770 \\ 
		scpb2 & (11.7) & 6661 & 576.6 & 95526 & 193.1 & 25416 & 190.0 & 98827 \\ 
		scpb3 & (11.5) & 4869 & 341.4 & 25726 & 90.1 & 10854 & 97.1 & 43021 \\ 
		scpb4 & (13.1) & 4553 & 720.0 & 96625 & 193.2 & 24709 & 246.3 & 153084 \\ 
		scpb5 & (10.8) & 4727 & 233.7 & 18253 & 47.4 & 3834 & 48.8 & 23480 \\ 
		scpc1 & (8.6) & 2202 & 573.7 & 32800 & 73.1 & 3829 & 115.4 & 30432 \\ 
		scpc2 & (8.2) & 1759 & 881.5 & 69243 & 351.4 & 29043 & 358.2 & 113709 \\ 
		scpc3 & (9.1) & 1423 & 2262.3 & 312112 & 778.3 & 75759 & 1829.2 & 501243 \\ 
		scpc4 & (5.7) & 2021 & 276.6 & 8374 & 36.6 & 2050 & 27.9 & 4645 \\ 
		scpc5 & (9.0) & 2337 & 1103.1 & 102990 & 166.2 & 15332 & 304.1 & 97395 \\ 
		scpd1 & (13.7) & 2664 & 652.8 & 41611 & 110.0 & 9014 & 108.9 & 34953 \\ 
		scpd2 & (13.3) & 2874 & 1452.5 & 154692 & 317.9 & 51405 & 437.0 & 226652 \\ 
		scpd3 & (14.6) & 2509 & 4063.0 & 714490 & 1092.1 & 186792 & 1624.0 & 804046 \\ 
		scpd4 & (15.6) & 2833 & 1553.4 & 163159 & 448.7 & 53332 & 953.9 & 477418 \\ 
		scpd5 & (14.5) & 2595 & 1236.6 & 128150 & 311.0 & 36591 & 413.3 & 200392 \\ 
		scpe1 & (23.8) & 202909 & 291.9 & 440769 & 20.0 & 8267 & 997.0 & 1209582 \\ 
		scpe2 & (14.9) & 124022 & 23.0 & 39773 & 20.9 & 7898 & 13.8 & 20017 \\ 
		scpe3 & (17.3) & 143209 & 28.0 & 44787 & 27.8 & 11285 & 22.8 & 36916 \\ 
		scpe4 & (15.5) & 244409 & 72.1 & 137334 & 569.5 & 182845 & 408.8 & 473928 \\ 
		scpe5 & (12.8) & 127191 & 21.7 & 35065 & 26.9 & 11054 & 15.8 & 25085 \\ 
		scpnre1 & (23.7) & 1887 & (9.9) & 618781 & (3.3) & 841592 & (4.4) & 2918483 \\ 
		scpnre2 & (26.7) & 1515 & (8.9) & 2117769 & (6.2) & 739139 & (6.5) & 2888883 \\ 
		scpnre3 & (28.0) & 1802 & (10.1) & 646693 & 3546.1 & 507681 & 6513.6 & 3802223 \\ 
		scpnre4 & (22.6) & 2335 & (9.7) & 600269 & (4.7) & 904683 & (5.2) & 3131983 \\ 
		scpnre5 & (24.1) & 1669 & (4.9) & 2802269 & 3531.8 & 490633 & (4.6) & 3309083 \\ 
		scpnrf1 & (29.5) & 1623 & (12.2) & 806303 & 1543.4 & 258465 & 2083.0 & 1524626 \\ 
		scpnrf2 & (25.4) & 3631 & 1385.3 & 946283 & 787.9 & 131735 & 940.8 & 878851 \\ 
		scpnrf3 & (28.8) & 2385 & 5485.7 & 2478715 & 2063.2 & 330887 & 2048.9 & 1414203 \\ 
		scpnrf4 & (28.5) & 2111 & 6998.8 & 2677316 & 5211.1 & 948167 & (4.6) & 5155605 \\ 
		scpnrf5 & (34.9) & 2431 & (16.7) & 2504969 & (11.1) & 804432 & (11.6) & 2339715 \\ 
	\end{longtable}
}
{
	\centering
	\mytabcolsepeach
	\myarraystretch
	\centering
	\scriptsize
	\begin{longtable}{rrrrrrrrrr}
		\caption{Performance comparison of settings \CPX, \AUTO, \SBD, and \RBD for the independent instances with $\epsilon =0.1$ in testset \Tone. 
			\T (\G \%) denotes that the CPU time is \T if the instance is solved within the timelimit; otherwise, it denotes that the end gap returned by \CPLEX is \G \%. $\NODE$ denotes the number of explored nodes.}\\
		\label{tab:table2}\\
		\headlinea
		\multirow{2}{*}{\texttt{Name}} & \multicolumn{2}{c}{\CPX} & \multicolumn{2}{c}{\AUTO} & \multicolumn{2}{c}{\SBD} & \multicolumn{2}{c}{\RBD}\\\cmidrule(r){2-3}\cmidrule(r){4-5}\cmidrule(r){6-7} \cmidrule(r){8-9} & {\T (\G \%)} & \multicolumn{1}{r}{\NODE} & {\T (\G \%)} & \multicolumn{1}{r}{\NODE} & \T (\G \%) & \multicolumn{1}{r}{\NODE} & \T (\G \%) & \multicolumn{1}{r}{\NODE}\\
		\headlineb
		\endhead
		\footlinea
		\multicolumn{9}{r}{continued on next page}\\
		\footlineb
		\endfoot
		\endline
		\endlastfoot
		scp41 & 365.5 & 1663 & 4.2 & 0 & 2.2 & 0 & 2.2 & 0 \\ 
		scp42 & 293.9 & 1341 & 4.3 & 0 & 2.1 & 0 & 2.2 & 0 \\ 
		scp43 & 554.8 & 2262 & 6.0 & 71 & 3.5 & 18 & 3.5 & 21 \\ 
		scp44 & 201.3 & 859 & 1.9 & 0 & 2.1 & 0 & 2.1 & 0 \\ 
		scp45 & 498.2 & 1843 & 4.8 & 35 & 3.6 & 7 & 3.5 & 7 \\ 
		scp46 & 738.2 & 3436 & 5.5 & 117 & 3.6 & 24 & 3.4 & 24 \\ 
		scp47 & 133.1 & 253 & 2.0 & 0 & 1.3 & 0 & 1.3 & 0 \\ 
		scp48 & 722.7 & 3316 & 5.5 & 215 & 3.8 & 43 & 3.6 & 43 \\ 
		scp49 & 776.2 & 2654 & 5.0 & 31 & 2.1 & 0 & 2.1 & 0 \\ 
		scp410 & 129.7 & 515 & 4.3 & 0 & 2.1 & 0 & 2.2 & 0 \\ 
		scp51 & 737.7 & 1926 & 3.5 & 0 & 2.2 & 0 & 2.2 & 0 \\ 
		scp52 & 3180.1 & 19408 & 8.3 & 145 & 4.8 & 78 & 4.3 & 62 \\ 
		scp53 & 877.5 & 4057 & 9.4 & 43 & 4.1 & 32 & 3.8 & 32 \\ 
		scp54 & 3885.5 & 16160 & 7.3 & 9 & 4.0 & 5 & 4.0 & 5 \\ 
		scp55 & 95.2 & 147 & 2.8 & 0 & 2.1 & 0 & 2.2 & 0 \\ 
		scp56 & 624.8 & 2354 & 6.5 & 0 & 2.4 & 0 & 2.4 & 0 \\ 
		scp57 & 517.3 & 1232 & 6.3 & 0 & 3.6 & 1 & 3.6 & 1 \\ 
		scp58 & 621.5 & 2120 & 4.9 & 0 & 2.1 & 0 & 2.1 & 0 \\ 
		scp59 & 765.8 & 3024 & 10.5 & 66 & 4.1 & 16 & 3.7 & 12 \\ 
		scp510 & 1016.4 & 2968 & 4.2 & 0 & 2.1 & 0 & 2.3 & 0 \\ 
		scp61 & 1139.4 & 4046 & 4.3 & 95 & 3.8 & 30 & 3.7 & 30 \\ 
		scp62 & 2596.6 & 13549 & 5.3 & 122 & 4.0 & 25 & 3.9 & 25 \\ 
		scp63 & (2.1) & 11783 & 4.9 & 12 & 3.8 & 9 & 3.5 & 9 \\ 
		scp64 & 1246.2 & 4040 & 6.5 & 47 & 3.9 & 32 & 3.8 & 27 \\ 
		scp65 & (2.2) & 25497 & 5.5 & 807 & 5.5 & 219 & 4.6 & 218 \\ 
		scpa1 & (5.5) & 4260 & 14.8 & 413 & 7.9 & 190 & 5.3 & 204 \\ 
		scpa2 & (1.7) & 7407 & 12.6 & 347 & 6.0 & 122 & 4.5 & 130 \\ 
		scpa3 & (2.9) & 7291 & 15.6 & 275 & 5.6 & 37 & 4.7 & 37 \\ 
		scpa4 & (1.6) & 7781 & 15.2 & 110 & 5.4 & 41 & 4.5 & 38 \\ 
		scpa5 & (4.2) & 5463 & 15.3 & 79 & 6.1 & 63 & 4.8 & 47 \\ 
		scpb1 & (9.9) & 5969 & 31.7 & 7180 & 21.7 & 3125 & 10.3 & 3780 \\ 
		scpb2 & (9.8) & 5669 & 28.0 & 1336 & 8.1 & 387 & 5.8 & 365 \\ 
		scpb3 & (11.3) & 4769 & 18.0 & 1405 & 8.2 & 554 & 5.8 & 564 \\ 
		scpb4 & (9.3) & 6669 & 29.1 & 3034 & 12.1 & 1201 & 7.0 & 1435 \\ 
		scpb5 & (6.7) & 7906 & 20.2 & 622 & 7.3 & 104 & 5.7 & 117 \\ 
		scpc1 & (7.8) & 2246 & 38.3 & 2083 & 14.5 & 898 & 6.6 & 605 \\ 
		scpc2 & (11.7) & 2067 & 26.7 & 432 & 9.2 & 107 & 6.9 & 105 \\ 
		scpc3 & (8.6) & 1873 & 59.3 & 7230 & 43.7 & 4178 & 14.7 & 3839 \\ 
		scpc4 & (8.3) & 2079 & 23.2 & 443 & 7.5 & 102 & 5.8 & 158 \\ 
		scpc5 & (8.0) & 1824 & 54.5 & 2323 & 13.0 & 531 & 6.9 & 727 \\ 
		scpd1 & (8.9) & 3363 & 43.3 & 3065 & 13.2 & 973 & 8.4 & 891 \\ 
		scpd2 & (11.2) & 3615 & 26.3 & 1776 & 11.3 & 372 & 8.3 & 415 \\ 
		scpd3 & (15.3) & 2632 & 51.6 & 3287 & 11.5 & 426 & 8.6 & 857 \\ 
		scpd4 & (12.9) & 3067 & 39.4 & 5273 & 16.9 & 1640 & 9.5 & 1600 \\ 
		scpd5 & (11.3) & 3182 & 52.8 & 6557 & 22.0 & 2048 & 11.4 & 2432 \\ 
		scpe1 & (21.7) & 181178 & 4.6 & 7594 & 7.9 & 5117 & 7.3 & 8169 \\ 
		scpe2 & 1349.6 & 14888 & 1.5 & 479 & 3.5 & 26 & 3.1 & 32 \\ 
		scpe3 & 709.0 & 8289 & 0.7 & 216 & 3.5 & 60 & 3.0 & 60 \\ 
		scpe4 & 1241.4 & 29164 & 0.8 & 14 & 3.5 & 11 & 3.1 & 11 \\ 
		scpe5 & (20.1) & 208082 & 3.9 & 6569 & 3.4 & 19 & 3.0 & 19 \\ 
		scpnre1 & (25.3) & 1622 & 228.1 & 158928 & 211.4 & 57880 & 101.3 & 53471 \\ 
		scpnre2 & (25.2) & 1616 & 487.7 & 425608 & 276.5 & 69704 & 219.8 & 118952 \\ 
		scpnre3 & (21.0) & 1127 & 736.0 & 72439 & 107.3 & 28256 & 62.1 & 30832 \\ 
		scpnre4 & (19.9) & 2437 & 163.3 & 70287 & 114.1 & 24105 & 56.3 & 26972 \\ 
		scpnre5 & (19.3) & 1729 & 160.7 & 55359 & 69.4 & 11914 & 44.7 & 17738 \\ 
		scpnrf1 & (28.6) & 2639 & 197.8 & 84989 & 91.8 & 29188 & 58.9 & 29588 \\ 
		scpnrf2 & (23.8) & 2131 & 124.0 & 34403 & 66.8 & 14467 & 38.1 & 14502 \\ 
		scpnrf3 & (23.5) & 1671 & 100.4 & 29090 & 61.2 & 7808 & 33.0 & 8400 \\ 
		scpnrf4 & (28.7) & 2369 & 314.0 & 214797 & 225.7 & 108052 & 124.8 & 91633 \\ 
		scpnrf5 & (30.9) & 2331 & 794.5 & 671607 & 796.2 & 403883 & 645.3 & 497981 \\ 
	\end{longtable}
}
{
	\centering
	\mytabcolsepeach
	\myarraystretch
	\centering
	\scriptsize
	\begin{longtable}{rrrrrrrrrr}
		\caption{Performance comparison of settings \CPX, \AUTO, \SBD, and \RBD for the \rev{correlated} instances with $\epsilon =0.05$ in testset \Tone.             
			\T (\G \%) denotes that the CPU time is \T if the instance is solved within the timelimit; otherwise, it denotes that the end gap returned by \CPLEX is \G \%. $\NODE$ denotes the number of explored nodes.}\\
		\label{tab:table3}\\
		\headlinea
		\multirow{2}{*}{\texttt{Name}} & \multicolumn{2}{c}{\CPX} & \multicolumn{2}{c}{\AUTO} & \multicolumn{2}{c}{\SBD} & \multicolumn{2}{c}{\RBD}\\\cmidrule(r){2-3}\cmidrule(r){4-5}\cmidrule(r){6-7} \cmidrule(r){8-9} & {\T (\G \%)} & \multicolumn{1}{r}{\NODE} & {\T (\G \%)} & \multicolumn{1}{r}{\NODE} & \T (\G \%) & \multicolumn{1}{r}{\NODE} & \T (\G \%) & \multicolumn{1}{r}{\NODE}\\
		\headlineb
		\endhead
		\footlinea
		\multicolumn{9}{r}{continued on next page}\\
		\footlineb
		\endfoot
		\endline
		\endlastfoot
		scp41 & 465.0 & 1372 & 6.5 & 434 & 4.3 & 70 & 4.0 & 212 \\ 
		scp42 & 426.0 & 1377 & 7.6 & 260 & 4.3 & 29 & 4.3 & 74 \\ 
		scp43 & 150.0 & 283 & 5.7 & 247 & 4.0 & 9 & 3.9 & 22 \\ 
		scp44 & 463.6 & 1813 & 6.0 & 526 & 4.3 & 98 & 4.1 & 306 \\ 
		scp45 & 504.4 & 1896 & 12.6 & 1619 & 7.1 & 328 & 5.6 & 1133 \\ 
		scp46 & 1221.1 & 2846 & 18.2 & 2796 & 6.0 & 207 & 7.8 & 1681 \\ 
		scp47 & 333.9 & 1366 & 6.0 & 215 & 4.1 & 31 & 3.9 & 67 \\ 
		scp48 & 483.2 & 1451 & 9.0 & 950 & 4.3 & 41 & 4.2 & 164 \\ 
		scp49 & 963.5 & 2362 & 11.4 & 938 & 4.3 & 49 & 4.2 & 164 \\ 
		scp410 & 747.3 & 1570 & 7.7 & 464 & 4.7 & 66 & 4.6 & 140 \\ 
		scp51 & 211.0 & 267 & 8.8 & 0 & 3.8 & 7 & 4.0 & 7 \\ 
		scp52 & 3402.0 & 13297 & 16.5 & 1058 & 5.8 & 104 & 5.1 & 788 \\ 
		scp53 & 1193.5 & 4214 & 19.9 & 2043 & 7.4 & 253 & 6.4 & 1570 \\ 
		scp54 & 1162.4 & 3870 & 11.4 & 694 & 5.5 & 105 & 4.9 & 246 \\ 
		scp55 & 84.0 & 59 & 9.4 & 35 & 4.4 & 8 & 4.2 & 14 \\ 
		scp56 & 304.4 & 657 & 12.9 & 339 & 5.5 & 43 & 4.3 & 179 \\ 
		scp57 & 768.0 & 2204 & 12.2 & 711 & 5.7 & 83 & 4.8 & 220 \\ 
		scp58 & 673.7 & 1551 & 16.9 & 693 & 5.5 & 64 & 4.5 & 121 \\ 
		scp59 & 1109.1 & 3464 & 16.5 & 911 & 6.2 & 79 & 5.2 & 282 \\ 
		scp510 & 852.4 & 2381 & 11.4 & 920 & 6.5 & 147 & 4.9 & 213 \\ 
		scp61 & 2928.7 & 11156 & 12.6 & 1944 & 6.6 & 270 & 5.6 & 1087 \\ 
		scp62 & 1549.5 & 7423 & 7.2 & 1902 & 5.5 & 235 & 5.4 & 1184 \\ 
		scp63 & (3.1) & 12542 & 17.1 & 4006 & 6.4 & 318 & 8.3 & 2397 \\ 
		scp64 & 2099.7 & 7740 & 9.3 & 1174 & 6.2 & 272 & 5.3 & 1016 \\ 
		scp65 & (1.2) & 16174 & 13.3 & 2992 & 7.0 & 301 & 7.8 & 2142 \\ 
		scpa1 & (4.3) & 5082 & 66.4 & 10572 & 38.7 & 2020 & 20.5 & 5354 \\ 
		scpa2 & (4.4) & 4070 & 92.0 & 8141 & 27.3 & 2072 & 30.6 & 9185 \\ 
		scpa3 & (1.6) & 6000 & 47.1 & 5355 & 15.1 & 563 & 13.0 & 2148 \\ 
		scpa4 & (5.1) & 3825 & 175.7 & 10990 & 26.5 & 1860 & 29.3 & 8787 \\ 
		scpa5 & 3848.4 & 5618 & 33.1 & 1288 & 7.9 & 172 & 6.5 & 514 \\ 
		scpb1 & (12.0) & 5169 & 731.5 & 131360 & 375.0 & 56091 & 362.4 & 172282 \\ 
		scpb2 & (13.0) & 4545 & 536.4 & 74148 & 126.5 & 13894 & 169.7 & 80927 \\ 
		scpb3 & (6.6) & 6614 & 38.8 & 2834 & 13.9 & 427 & 9.3 & 1315 \\ 
		scpb4 & (8.1) & 6233 & 57.2 & 10093 & 20.2 & 1146 & 19.8 & 6862 \\ 
		scpb5 & (9.3) & 5997 & 350.2 & 27663 & 22.6 & 1945 & 27.5 & 12219 \\ 
		scpc1 & (8.0) & 1637 & 1227.6 & 95782 & 152.6 & 10259 & 169.0 & 43148 \\ 
		scpc2 & (8.1) & 2014 & 861.1 & 51165 & 117.8 & 6334 & 171.1 & 40938 \\ 
		scpc3 & (10.7) & 1639 & 1166.2 & 105661 & 286.0 & 26421 & 427.3 & 134194 \\ 
		scpc4 & (9.2) & 1930 & 1812.1 & 147473 & 177.2 & 11328 & 446.8 & 119519 \\ 
		scpc5 & (8.8) & 1939 & 2860.1 & 211377 & 452.0 & 46089 & 493.3 & 133370 \\ 
		scpd1 & (18.4) & 2393 & 1222.4 & 603106 & 566.5 & 70259 & 1198.7 & 540349 \\ 
		scpd2 & (17.2) & 2703 & 2619.0 & 444077 & 1583.5 & 235201 & 1718.7 & 839856 \\ 
		scpd3 & (11.7) & 2553 & 2149.0 & 310118 & 556.9 & 74209 & 1181.2 & 558066 \\ 
		scpd4 & (15.2) & 2594 & 1584.3 & 143756 & 421.1 & 51913 & 289.8 & 119330 \\ 
		scpd5 & (14.4) & 2842 & 1519.5 & 159509 & 360.9 & 41688 & 499.5 & 173137 \\ 
		scpe1 & (24.2) & 210709 & 79.2 & 159559 & 232.6 & 89478 & 15.8 & 20447 \\ 
		scpe2 & 6559.2 & 135756 & 16.5 & 34516 & 19.1 & 7179 & 13.4 & 20529 \\ 
		scpe3 & (18.3) & 142223 & 29.0 & 58052 & 29.8 & 11153 & 21.9 & 36547 \\ 
		scpe4 & 3837.9 & 74155 & 9.3 & 17571 & 19.7 & 7596 & 10.2 & 14017 \\ 
		scpe5 & (14.1) & 114503 & 334.8 & 469452 & 22.4 & 9351 & 15.7 & 26985 \\ 
		scpnre1 & (26.4) & 1942 & (7.8) & 413177 & 3042.0 & 424957 & 3582.8 & 1899603 \\ 
		scpnre2 & (26.7) & 1797 & (8.4) & 2290269 & (7.5) & 652883 & (7.8) & 2649383 \\ 
		scpnre3 & (24.3) & 2188 & (7.6) & 505714 & (3.7) & 817621 & (4.4) & 3292283 \\ 
		scpnre4 & (22.1) & 2053 & (9.2) & 610009 & (5.6) & 830283 & (6.1) & 2890283 \\ 
		scpnre5 & (25.4) & 1271 & (6.8) & 2628569 & 6103.5 & 835667 & (6.7) & 2770183 \\ 
		scpnrf1 & (27.3) & 2431 & 1795.1 & 1116438 & 2516.2 & 590110 & 2399.5 & 2005720 \\ 
		scpnrf2 & (22.7) & 4931 & 1264.8 & 829410 & 1072.4 & 190599 & 902.5 & 802042 \\ 
		scpnrf3 & (26.6) & 3771 & 626.5 & 340005 & 588.7 & 101540 & 511.2 & 376534 \\ 
		scpnrf4 & (33.7) & 4037 & (12.2) & 2748769 & 5068.6 & 918957 & 5333.5 & 4163038 \\ 
		scpnrf5 & (35.5) & 2031 & (13.0) & 2637669 & (6.1) & 935515 & (12.1) & 2974415 \\ 
	\end{longtable}
}
{
	\centering
	\mytabcolsepeach
	\myarraystretch
	\centering
	\scriptsize
	\begin{longtable}{rrrrrrrrrr}
		\caption{Performance comparison of settings \CPX, \AUTO, \SBD, and \RBD for the \rev{correlated} instances with $\epsilon =0.1$ in testset \Tone.             
			\T (\G \%) denotes that the CPU time is \T if the instance is solved within the timelimit; otherwise, it denotes that the end gap returned by \CPLEX is \G \%. $\NODE$ denotes the number of explored nodes.}\\
		\label{tab:table4}\\
		\headlinea
		\multirow{2}{*}{\texttt{Name}} & \multicolumn{2}{c}{\CPX} & \multicolumn{2}{c}{\AUTO} & \multicolumn{2}{c}{\SBD} & \multicolumn{2}{c}{\RBD}\\\cmidrule(r){2-3}\cmidrule(r){4-5}\cmidrule(r){6-7} \cmidrule(r){8-9} & {\T (\G \%)} & \multicolumn{1}{r}{\NODE} & {\T (\G \%)} & \multicolumn{1}{r}{\NODE} & \T (\G \%) & \multicolumn{1}{r}{\NODE} & \T (\G \%) & \multicolumn{1}{r}{\NODE}\\
		\headlineb
		\endhead
		\footlinea
		\multicolumn{9}{r}{continued on next page}\\
		\footlineb
		\endfoot
		\endline
		\endlastfoot
		scp41 & 207.6 & 1036 & 6.0 & 31 & 3.6 & 15 & 3.4 & 15 \\ 
		scp42 & 294.9 & 1207 & 1.8 & 0 & 2.1 & 0 & 2.1 & 0 \\ 
		scp43 & 287.9 & 1433 & 5.9 & 46 & 3.5 & 3 & 3.3 & 6 \\ 
		scp44 & 469.2 & 1698 & 1.7 & 0 & 2.2 & 0 & 2.2 & 0 \\ 
		scp45 & 193.0 & 851 & 2.8 & 0 & 2.1 & 0 & 2.1 & 0 \\ 
		scp46 & 320.2 & 1037 & 3.6 & 0 & 2.2 & 0 & 2.2 & 0 \\ 
		scp47 & 87.8 & 114 & 1.7 & 0 & 2.1 & 0 & 2.0 & 0 \\ 
		scp48 & 820.9 & 3576 & 5.4 & 183 & 3.9 & 40 & 3.7 & 41 \\ 
		scp49 & 641.2 & 2470 & 4.0 & 0 & 2.1 & 0 & 2.1 & 0 \\ 
		scp410 & 102.1 & 171 & 1.8 & 0 & 2.1 & 0 & 2.1 & 0 \\ 
		scp51 & 595.8 & 1936 & 6.2 & 0 & 2.2 & 0 & 2.3 & 0 \\ 
		scp52 & (2.0) & 26268 & 6.7 & 185 & 5.0 & 91 & 4.3 & 136 \\ 
		scp53 & 650.3 & 2532 & 8.3 & 0 & 2.5 & 0 & 2.4 & 0 \\ 
		scp54 & 3098.5 & 16503 & 9.5 & 23 & 4.1 & 35 & 3.8 & 40 \\ 
		scp55 & 97.6 & 122 & 2.2 & 0 & 2.2 & 0 & 2.1 & 0 \\ 
		scp56 & 682.2 & 3019 & 6.2 & 0 & 2.2 & 0 & 2.1 & 0 \\ 
		scp57 & 611.2 & 2358 & 10.8 & 71 & 4.5 & 25 & 4.1 & 27 \\ 
		scp58 & 455.2 & 1861 & 3.4 & 0 & 2.2 & 0 & 2.3 & 0 \\ 
		scp59 & 962.9 & 3309 & 10.2 & 62 & 4.4 & 68 & 3.9 & 70 \\ 
		scp510 & 836.6 & 2338 & 9.9 & 5 & 3.9 & 7 & 3.8 & 7 \\ 
		scp61 & 2137.2 & 11661 & 5.7 & 255 & 4.7 & 176 & 4.1 & 122 \\ 
		scp62 & 2539.5 & 9760 & 6.8 & 108 & 4.0 & 40 & 3.8 & 44 \\ 
		scp63 & 1743.8 & 4370 & 5.2 & 15 & 3.6 & 9 & 3.6 & 9 \\ 
		scp64 & 2575.1 & 11502 & 6.3 & 34 & 3.8 & 19 & 3.8 & 19 \\ 
		scp65 & (1.6) & 24010 & 6.8 & 639 & 5.0 & 174 & 4.3 & 243 \\ 
		scpa1 & (3.1) & 5831 & 14.0 & 682 & 7.1 & 196 & 5.1 & 214 \\ 
		scpa2 & (2.4) & 7493 & 15.0 & 121 & 5.7 & 40 & 4.9 & 51 \\ 
		scpa3 & (4.5) & 6633 & 14.8 & 253 & 6.0 & 46 & 4.9 & 36 \\ 
		scpa4 & (5.5) & 5530 & 17.2 & 267 & 6.7 & 144 & 5.0 & 71 \\ 
		scpa5 & (1.7) & 8369 & 12.5 & 71 & 4.9 & 21 & 4.6 & 23 \\ 
		scpb1 & (13.2) & 5298 & 29.2 & 6610 & 20.0 & 3247 & 8.3 & 3072 \\ 
		scpb2 & (12.3) & 4569 & 27.6 & 885 & 6.9 & 86 & 5.8 & 219 \\ 
		scpb3 & (10.6) & 5909 & 27.4 & 2604 & 11.0 & 1355 & 5.8 & 671 \\ 
		scpb4 & (9.9) & 6605 & 30.6 & 1522 & 12.8 & 1098 & 6.3 & 960 \\ 
		scpb5 & (9.7) & 6669 & 19.7 & 826 & 8.6 & 412 & 6.5 & 476 \\ 
		scpc1 & (8.0) & 2127 & 43.8 & 2236 & 12.5 & 447 & 7.8 & 676 \\ 
		scpc2 & (6.0) & 2238 & 49.1 & 2364 & 18.9 & 1177 & 7.3 & 1116 \\ 
		scpc3 & (10.7) & 1887 & 81.8 & 14341 & 63.1 & 8151 & 22.5 & 7361 \\ 
		scpc4 & (7.9) & 1558 & 26.5 & 528 & 8.2 & 157 & 5.9 & 122 \\ 
		scpc5 & (8.6) & 2275 & 44.5 & 2707 & 16.2 & 1396 & 7.1 & 1049 \\ 
		scpd1 & (12.7) & 3049 & 24.2 & 1121 & 12.5 & 595 & 7.7 & 514 \\ 
		scpd2 & (13.2) & 2915 & 41.5 & 4893 & 11.3 & 317 & 8.0 & 270 \\ 
		scpd3 & (11.1) & 2531 & 33.6 & 2068 & 13.3 & 812 & 8.5 & 874 \\ 
		scpd4 & (14.2) & 2633 & 51.6 & 7123 & 18.5 & 1183 & 9.7 & 1406 \\ 
		scpd5 & (12.5) & 2877 & 41.3 & 2717 & 15.7 & 1063 & 8.5 & 730 \\ 
		scpe1 & (22.0) & 165023 & 4.0 & 6056 & 7.8 & 5758 & 6.2 & 6113 \\ 
		scpe2 & 628.0 & 6694 & 0.7 & 75 & 3.3 & 36 & 3.1 & 27 \\ 
		scpe3 & 6625.8 & 194620 & 0.7 & 158 & 3.4 & 78 & 3.1 & 93 \\ 
		scpe4 & 4854.2 & 162737 & 0.8 & 18 & 3.3 & 16 & 3.2 & 16 \\ 
		scpe5 & (21.4) & 175247 & 20.4 & 39241 & 11.7 & 9544 & 17.5 & 27729 \\ 
		scpnre1 & (25.5) & 1482 & 272.4 & 167386 & 165.8 & 53354 & 111.4 & 65132 \\ 
		scpnre2 & (25.6) & 1322 & 421.6 & 306212 & 256.4 & 92581 & 206.5 & 123987 \\ 
		scpnre3 & (20.7) & 2141 & 190.4 & 99597 & 115.6 & 26685 & 62.0 & 31303 \\ 
		scpnre4 & (26.5) & 1739 & 159.5 & 59464 & 107.0 & 22704 & 46.5 & 19717 \\ 
		scpnre5 & (20.9) & 1492 & 129.1 & 70924 & 81.1 & 15504 & 42.4 & 15901 \\ 
		scpnrf1 & (26.6) & 1703 & 201.3 & 78791 & 93.8 & 25706 & 52.5 & 24857 \\ 
		scpnrf2 & (22.9) & 2625 & 282.1 & 31782 & 54.4 & 12568 & 38.9 & 12146 \\ 
		scpnrf3 & (23.2) & 1591 & 123.9 & 24673 & 56.3 & 8625 & 33.6 & 7180 \\ 
		scpnrf4 & (29.4) & 2248 & 282.8 & 161728 & 185.8 & 75329 & 128.6 & 88579 \\ 
		scpnrf5 & (26.1) & 3425 & 991.6 & 969952 & 793.0 & 481697 & 600.5 & 480505 \\ 
	\end{longtable}
}

	\bibliography{BD_PSCP}

\end{document}